\documentclass[a4paper]{amsart}
\usepackage[utf8]{inputenc}

\usepackage[a4paper,width={16cm},left=2.5cm,bottom=3cm, top=3cm,textwidth=455.24408pt,textheight=674.33032pt]{geometry}

\usepackage{latexsym}
\usepackage{graphics}
\usepackage{amsmath}
\usepackage{amssymb}
\usepackage{bm}
\usepackage{mathtools}
\usepackage{mathrsfs}
\usepackage{enumitem}
\usepackage{csquotes}
\usepackage{pgf,tikz}
\usepackage[normalem]{ulem}

\usepackage{hyperref}
\usepackage{cleveref}
\hypersetup{
	linktocpage=true,
	colorlinks=true,
	linkcolor=blue!70!black,
	citecolor=orange!40!red,
	urlcolor=magenta!80!black,
}

\usepackage{graphicx}
\usepackage{color}

\numberwithin{equation}{section} 

\newtheorem{theorem}{Theorem}[section]
\newtheorem{corollary}[theorem]{Corollary}
\newtheorem{lemma}[theorem]{Lemma}
\newtheorem{proposition}[theorem]{Proposition}
\theoremstyle{definition}

\newcommand{\R}{\mathbb{R}}	
\newcommand{\N}{\mathbb{N}} 
\newcommand{\dx}{\,\mathrm{d}x}	
\newcommand{\ds}{\,\mathrm{d}S}	

\newcommand{\e}{\varepsilon}	
\newcommand{\weak}{\rightharpoonup}
\newcommand{\nnu}{{\bm{\nu}}}  

\newcommand{\norm}[1]{\left\lVert #1 \right\lVert}
\newcommand{\abs}[1]{\left| #1 \right|}
\newcommand{\sub}{\subseteq}
\newcommand\restr[1]{\raisebox{-.5ex}{$|$}_{#1}}

\newcommand{\tu}[1]{\textup{#1}}

\DeclareMathOperator{\dist}{\mathrm{dist}}

\newenvironment{bvp}{\left\{\begin{aligned}  }{\end{aligned}\right.}

\title[On the splitting of Neumann Eigenvalues in Perforated
Domains]{On the splitting of Neumann Eigenvalues \\ in Perforated
	Domains}

\author{Veronica Felli, Lorenzo Liverani and Roberto Ognibene}

\address{Veronica Felli and Roberto Ognibene
	\newline \indent Dipartimento di Matematica e Applicazioni
	\newline \indent
	Universit\`a degli Studi di Milano–Bicocca
	\newline\indent Via Cozzi 55, 20125 Milano, Italy}
\email{veronica.felli@unimib.it, roberto.ognibene@unimib.it}

\address{Lorenzo Liverani
	\newline \indent Department of Mathematics
	\newline \indent FAU Erlangen-N{\"u}rnberg
	\newline\indent  Cauerstr. 11, 91058 Erlangen (Germany)}
\email{lorenzo.liverani@fau.de}



\begin{document}

	\begin{abstract}
		We address the problem of splitting of eigenvalues of the Neumann
		Laplacian under singular domain perturbations. We consider a domain
		perturbed by the excision of a small spherical hole shrinking to an
		interior point. Our main result establishes that the splitting of
		multiple eigenvalues is a generic property: if the center of the
		hole is located outside a set of Hausdorff dimension $N-1$ and the
		radius is sufficiently small, multiple eigenvalues split into
		branches of lower multiplicity. The proof relies on the validity of
		an asymptotic expansion for the perturbed eigenvalues in terms of
		the scaling parameter. Such an asymptotic formula is of independent
		interest and generalizes previous results; notably, in dimension
		$N\geq 3$, it is valid for holes of arbitrary shape.
	\end{abstract}
	
	\maketitle
	

	\section{Introduction}
	
	\noindent
	The study of eigenvalues of differential operators plays a
	pivotal role in the theory of partial differential equations,
	with significant impacts on various applied disciplines, ranging from
	physics to chemistry and engineering. Indeed, eigenvalues and
	eigenfunctions are essential in the analysis of the stability,
	resonance, and dynamic behavior of complex systems.
	
	A vast literature is dedicated to the
	study of the spectral response to perturbations, including, for
	instance, small variations of the domain. The present paper fits
	into this context, with its primary aim being the further
	advancement of the investigation initiated in \cite{FLO},
	regarding spectral stability of the Neumann Laplacian in perforated
	domains.
	
	In the perturbative spectral theory for the Dirichlet Laplacian in
	domains with small holes, the relevant quantity in the asymptotic
	expansion of the eigenvalue variation turns out to be the capacity of
	the removed set, see the papers \cite{Courtois1995} and \cite{AFHL},
	as well as their extension to polyharmonic operators carried out in
	\cite{FR2023}.  If, instead, Neumann boundary conditions are imposed
	on the hole, a richer and more complex phenomenology can be
	observed. Following the research by Rauch and Taylor \cite{Rauch1975},
	which establishes conditions for the stability of the Neumann
	spectrum, several studies explored the asymptotic behaviour of
	perturbed eigenvalues, see e.g. \cite{ozawa1983}, \cite{LNS2011},
	\cite{LdC2012}, \cite{jimbo}, \cite{bucur_GAFA}. The recent paper
	\cite{FLO} analyzes the asymptotic behavior of simple eigenvalues of
	the Neumann Laplacian in domains perforated by small interior
	holes. In that work, the leading-order term in the expansion is
	expressed through a geometric quantity analogous to (boundary)
	torsional rigidity, which plays a role in the Neumann setting similar
	to that of capacity in the Dirichlet case.
	
	In the present paper, we address the issue of genericity of simplicity
	for the Neumann eigenvalues in a perforated domain. In more detail,
	under Neumann conditions on both the external and the hole's boundary,
	we are interested in understanding whether multiple eigenvalues split
	into simple ones (or, in any case, in eigenbranches of lower
	multiplicity) if the domain is singularly perturbed, and, if so, how
	this property depends on the point at which the hole is excised. To
	this end, we first analyze the eigenbranches as the hole shrinks to a
	point, deriving an asymptotic expansion for the variation of multiple
	eigenvalues. Then we investigate the circumstances under which the
	eigenbranches are separated from each other: in the case of double
	eigenvalues and spherical-shaped holes, we show that this occurs for
	almost every position of the hole. Incidentally, our techniques can be
	seamlessly applied to the case of Dirichlet boundary conditions,
	thereby recovering some results on the splitting of Dirichlet
	eigenvalues previously established by Flucher in \cite{flucher}.  This
	is discussed in the Appendix. We mention that the proof by Flucher is
	based on Courant’s min--max principle, together with the selection of
	appropriate perturbed eigenfunctions used to estimate the Rayleigh
	quotient and constructed through suitable harmonic correction methods.
	
	The persistence of multiplicity of eigenvalues with respect to
	perturbations is a classical and challenging problem. The first
	results in this direction date back to the 1970s, in particular to the
	contributions of A. M. Micheletti
	\cite{micheletti1,micheletti2,micheletti3,micheletti4}, K. Uhlenbeck
	\cite{Uhlenbeck1,Uhlenbeck2} and J. H. Albert \cite{albert1,albert2}.
	The results contained in these works support the conclusion that the
	simplicity of the entire spectrum of a differential operator is a
	\emph{generic property}. In general, this means that for an operator
	with multiple eigenvalues, \enquote{almost every} sufficiently small
	perturbation results in an operator with only simple eigenvalues. A
	common feature of all the works mentioned above (which treat various
	types of perturbations) is that they all deal with regular
	perturbations of the eigenvalue problem, in the sense that there
	exists a diffeomorphism mapping each perturbed problem into the
	unperturbed one.  Further refinements of these results are contained
	in \cite{teytel}. The proof of the generic simplicity of the spectrum
	has been extended to other contexts, among which we mention (without
	claiming to provide an exhaustive list): \cite{bleecker}, which treats
	perturbations of the metric in a Riemannian manifold, \cite{stokes}
	which deals with the Stokes operator, \cite{luzzini-zaccaron}
	concerning the Maxwell eigenvalues, and \cite{fall,fall2} in which the
	authors investigate splitting and persistence of multiplicities for
	eigenvalues of the fractional Laplacian. In the context of singular
	perturbations, however, the literature is still quite limited.  One
	notable result is due to \cite{flucher} (see also
	\cite{dabrwoski1,ALM2022}), which provides a formula for the
	asymptotic expansion of Dirichlet eigenvalues in domains with small
	holes. As a consequence, one can deduce the splitting of any multiple
	eigenvalue, that is its ramification into eigenbranches of lower
	multiplicity (see \Cref{appendix} for further details).  To the best
	of our knowledge, the present work is the first to address the
	splitting of Neumann eigenvalues under singular perturbations.

	\subsection{Main Results}
	We deal with the Neumann eigenvalue problem in
	a bounded, connected, and open Lipschitz set $\Omega\sub\R^N$, $N\geq2$, 
	\begin{equation}\label{eq:unperturbed}
		\begin{bvp}
			-\Delta \varphi +\varphi &= \lambda\varphi, &&\text{in }\Omega, \\
			\partial_{\nnu}\varphi &=0, &&\text{on }\partial\Omega,
		\end{bvp}
	\end{equation}
	where $\nnu$ denotes the outer unit normal vector to the boundary
	$\partial\Omega$. We perturb the above problem by excising a small hole
	from the interior of $\Omega$.
	
	Let $\Sigma\sub\R^N$ be an open, bounded, and Lipschitz set, such
	that
	\begin{equation}\label{eq:connect}
		\overline{\Sigma}\subset B_{r_0}\quad\text{and}\quad
		B_{r_0}\setminus  \overline{\Sigma}\text{ is connected},
	\end{equation}
	for some $r_0>0$, where $B_{r_0}=\{x\in\R^N:|x|<r_0\}$. We define
	\begin{equation}\label{eq:delSigma-eps}
		\Sigma_\e^{x_0}:=x_0+\e\Sigma,
	\end{equation}
	where $x_0\in \Omega$ and $\e>0$ is a small perturbation parameter.
	More precisely, letting
	\begin{equation}\label{eq:choose-eps-zero}
		\e_0=\e_0(x_0,\Omega,\Sigma) :=\frac{\mathop{\rm dist}(x_0,\partial\Omega)}{1+r_0}>0,
	\end{equation}
	we observe that $\dist(\Sigma_\e^{x_0},\partial\Omega)\geq\e_0$ for all 
	$0<\e<\e_0$, so that
	\begin{equation*}
		\overline{\Sigma_\e^{x_0}}\subset\Omega \quad\text{for all $0<\e<\e_0$}.
	\end{equation*}
	We point out that assumption \eqref{eq:connect}, which was
	implicit in \cite{FLO}, is needed
	to ensure a uniform extension property for Sobolev
	spaces on domains with scaling holes.
	We refer to \cite{SW99} and to the appendix of \cite{FSR} for a
	proof of such an extension result
	and for an example showing how it fails in the absence of condition \eqref{eq:connect}.
	
	Letting
	\begin{equation*}
		\Omega_\e^{x_0,\Sigma}:=\Omega\setminus \overline{\Sigma_\e^{x_0}},
	\end{equation*}
	we consider the following perturbed problem
	\begin{equation}\label{eq:perturbed}\tag{$\mathcal P_\e^{x_0,\Sigma}$}
		\begin{bvp}
			-\Delta \varphi +\varphi &= \lambda\varphi, &&\text{in }\Omega_\e^{x_0,\Sigma}, \\
			\partial_{\nnu}\varphi &=0, &&\text{on
			}\partial\Omega_\e^{x_0,\Sigma}.
		\end{bvp}
	\end{equation}
	We denote by $\{\lambda_k(\Omega)\}_{k\in\N}$ the eigevalues of the
	unperturbed problem \eqref{eq:unperturbed} (repeated according to
	their multiplicities), by $\{\varphi_k\}_{k\in\N}$ a corresponding
	sequence of $L^2(\Omega)$-orthonormal eigenfunctions, and by
	$E(\lambda_k(\Omega))\sub H^1(\Omega)$ the eigenspace associated to
	the eigenvalue $\lambda_k(\Omega)$.  We note that
	$\lambda_0(\Omega) = 1$ for any domain $\Omega$, so it remains
	constant under domain perturbations and is therefore irrelevant for
	our analysis. Furthermore $\lambda_0(\Omega)=1$ is simple, so that
	$\lambda_k(\Omega)>1$ for every $k\geq1$.

	Analogously, we denote by
	$\{\lambda_k(\Omega_\e^{x_0,\Sigma})\}_{k\in\N}$ the perturbed
	eigenvalues, i.e. the eigenvalues of \eqref{eq:perturbed}, and by
	$E(\lambda_k(\Omega_\e^{x_0,\Sigma}))\sub H^1(\Omega_\e^{x_0,\Sigma})$
	the eigenspace corresponding to $\lambda_k(\Omega_\e^{x_0,\Sigma})$.
	For every $k\geq1$, we also denote as
	\begin{equation*}
		\mathfrak{m}_{k}^{x_0,\Sigma} (\e)=\dim \big(E(\lambda_k(\Omega_\e^{x_0,\Sigma}))\big)
	\end{equation*}
	the multiplicity of $\lambda_{k}(\Omega_\e^{x_0,\Sigma})$
	as an eigenvalue of \eqref{eq:perturbed}.

	As shown in the foundational work \cite{Rauch1975}, the spectrum
	of Neumann Laplacian is stable with respect to the domain
	perturbation introduced above, in the sense that
	\begin{equation}\label{eq:conveig}
		\lambda_k(\Omega_\e^{x_0,\Sigma})\to\lambda_k(\Omega)\quad
		\text{as }\e\to0,~\text{for all }k\in\N.
	\end{equation}
	We observe that, for every $k\in \N$,
	$\lambda_k(\Omega)
	=\lambda_k(\Omega\setminus\{x_0\})=\lambda_k(\Omega_0^{x_0,\Sigma})$
	due to the fact that the singleton $\{x_0\}$ has null capacity. 
	
	The rate of convergence in \eqref{eq:conveig} has been the
	subject of extensive quantitative analysis, including the recent
	work \cite{FLO}. In contrast, the \emph{qualitative} behavior of the
	spectrum under perturbations remains less explored. In particular,
	when $\lambda_k(\Omega)$ is a multiple eigenvalue, it is natural to
	ask whether its multiplicity is preserved under perturbation. Our
	main result addresses precisely this question.
	
	Throughout the paper we consider a fixed eigenvalue
	$\lambda_n(\Omega)$ of \eqref{eq:unperturbed}. Let
	\begin{equation*}
		m:=\dim E(\lambda_n(\Omega))
	\end{equation*}
	denote its
	multiplicity; hence, the index $n$ can be chosen in such a way that
	\begin{equation}\label{eq:scelta-n}
		\lambda_{n-1}(\Omega)<\lambda_n(\Omega)=\dots
		=\lambda_{n+m-1}(\Omega)<\lambda_{n+m}(\Omega).
	\end{equation}
	Since $\lambda_0(\Omega)=1$ is simple, it follows that
	\begin{equation}\label{eq:lambda_n>1}
		\text{if $m=\dim E(\lambda_n(\Omega))>1$,
			then $\lambda_n(\Omega)>1$}.
	\end{equation}
	In the case of a spherical hole $\Sigma=B_1$, our main result shows that the
	perturbed eigenvalues emanating from the multiple eigenvalue
	$\lambda_n(\Omega)$ all have multiplicity strictly less than
	$m=~\!\dim E(\lambda_n(\Omega))$.
	\begin{theorem}\label{thm:3} Let
		$N\geq 2$ and $m>1$.  There exists a relatively closed set
		$\Gamma\sub\Omega$ such that $\dim_{\mathcal{H}}\Gamma\leq N-1$
		and the following holds (with $\dim_{\mathcal{H}}\Gamma$
		denoting the Hausdorff dimension of $\Gamma$): for every
		$x_0\in\Omega\setminus\Gamma$ there exists
		$\e_{\textup{s}}=\e_{\textup{s}}(x_0)>0$ such that
		\begin{equation*}
			\mathfrak{m}_{n+i-1}^{x_0,B_1} (\e)<m\quad\text{for all
				$\e\in(0,\e_{\textup{s}})$ and $i=1,\dots,m$}.
		\end{equation*}
	\end{theorem}
	As a direct byproduct of Theorem \ref{thm:3} we obtain the following
	genericity result for the splitting of double eigenvalues.
	
	\begin{corollary}\label{cor:thm3} Let $N\geq 2$ and $m=2$. 
		There exists a relatively closed set $\Gamma\sub\Omega$ such
		that $\dim_{\mathcal{H}}\Gamma\leq N-1$ and the
		following holds: for every $x_0\in\Omega\setminus\Gamma$ there
		exists $\e_{\textup{s}}=\e_{\textup{s}}(x_0)>0$ such that
		\begin{equation*}
			\lambda_n(\Omega_\e^{x_0,B_1})<\lambda_{n+1}(\Omega_\e^{x_0,B_1})
		\end{equation*}
		for all $0<\e\leq \e_{\textup{s}}$. In particular, for every
		$x_0\in\Omega\setminus\Gamma$, $\lambda_n(\Omega_\e^{x_0,B_1})$ and
		$\lambda_{n+1}(\Omega_\e^{x_0,B_1})$ are simple, provided $\e$ is
		sufficiently small.
	\end{corollary}

	
	\subsection{Strategy of Proof}
	
	To prove Theorem \ref{thm:3}, the starting point is the general
	asymptotic expansion for Neumann eigenvalues in singularly perturbed
	domains, established in Theorem \ref{thm:1}. This is novel in the
	context of the Neumann Laplacian. In fact, all the auxiliary results
	employed in the proof also appear to be new and complement well the
	analysis brought forth in \cite{FLO}. We mention that asymptotic
	formulas for the eigenvalue variation are results of independent
	interest and may find applications, for instance, in non-existence of
	optimal shapes (see e.g. \cite{bucur_GAFA}).
	
	In the case $N\geq3$, we recall from \cite{FLO} the following notion
	of torsional rigidity: for any $f\in L^2(\partial\Sigma)$
	\begin{equation*}
		\tau_{\R^N\setminus\Sigma}(\partial\Sigma,f):=
		-2\min\left\{\frac{1}{2}\int_{\R^N\setminus
			\Sigma}
		|\nabla u|^2\dx-\int_{\partial\Sigma}fu\ds\colon u
		\in \mathcal{D}^{1,2}(\R^N\setminus\Sigma)\right\},
	\end{equation*}
	where the space $\mathcal{D}^{1,2}(\R^N\setminus \Sigma)$ is defined
	as the completion of $C_c^\infty(\R^N\setminus \Sigma)$ with respect
	to the norm
	\begin{equation*}
		\norm{u}_{\mathcal{D}^{1,2}(\R^N\setminus \Sigma)}
		:=\bigg(\int_{\R^N\setminus \Sigma} \abs{\nabla u}^2\dx \bigg)^{\!\frac{1}{2}}.
	\end{equation*}
	We also denote by
	$U_{\Sigma,f}\in\mathcal{D}^{1,2}(\R^N\setminus\Sigma)$ the unique
	minimizer attaining
	$\tau_{\R^N\setminus\Sigma}(\partial\Sigma,f)$. Equivalently,
	$U_{\Sigma,f}$ is the unique weak solution to
	\begin{equation*}
		\begin{dcases}
			-\Delta U = 0, &\text{in } \R^N\setminus \overline{\Sigma},\\
			\partial_\nu U = f, &\text{on } \partial\Sigma,
		\end{dcases}
	\end{equation*}
	so that
	\begin{equation*}
		\tau_{\R^N\setminus\Sigma}(\partial\Sigma,f):=-\int_{\R^N\setminus
			\Sigma}
		|\nabla U_{\Sigma,f}|^2\dx+2\int_{\partial\Sigma}f U_{\Sigma,f}\ds
		=\int_{\R^N\setminus
			\Sigma}
		|\nabla U_{\Sigma,f}|^2\dx.
	\end{equation*}
	The following theorem describes the asymptotic behavior of
	eigenbranches departing from the fixed eigenvalue
	$\lambda_n(\Omega)$ of \eqref{eq:unperturbed}, as the the domain is
	perturbed by the removal of the small hole $\Sigma_\e^{x_0}$
	shrinking to $x_0$ as $\e\to 0$.
	
	\begin{theorem}\label{thm:1}
		Let $N\geq 3$ and $x_0\in\Omega$.  Let
		$\{\ell_i^{x_0,\Sigma}\}_{i=1,\dots,m}$ be the eigenvalues (in
		descending order) of the
		bilinear form
		\begin{multline}\label{eq:limit-bi-form}
			\mathcal{L}_{x_0,\Sigma}(\varphi,\psi):=\int_{\R^N\setminus\Sigma}\nabla
			U_{\Sigma,\nabla\varphi(x_0)\cdot\nnu}\cdot\nabla
			U_{\Sigma,\nabla\psi(x_0)\cdot\nnu}\dx
			\\+|\Sigma|\left(\nabla\varphi(x_0)\cdot\nabla\psi(x_0)-
			(\lambda_n(\Omega)-1)\varphi(x_0)\psi(x_0)\right),
		\end{multline}
		defined for every $\varphi,\psi\in E(\lambda_n(\Omega))$. Then, for
		every $i=1,\dots,m$ we have
		\begin{equation*}
			\lambda_{n+i-1}\left(\Omega_\e^{x_0,\Sigma}\right)=
			\lambda_n(\Omega)-\ell_i^{x_0,\Sigma}\e^N+o(\e^N)\quad\text{as }\e\to 0.
		\end{equation*}
	\end{theorem}

	The proof of Theorem \ref{thm:1} is based of the \emph{Lemma on small
		eigenvalues}, due to Colin de Verdière \cite{ColindeV1986} and
	revisited in \cite{Courtois1995} and \cite{ALM2022}, combined with a
	blow-up analysis for the torsion function \eqref{eq:torfun} in the
	spirit of \cite{FLO}.  In fact, building upon the same results,
	it is also possible to obtain an asymptotic expansion of the energy
	difference between any eigenfunction associated with
	$\lambda_n(\Omega)$ and its projection onto the direct sum of the
	corresponding perturbed eigenspaces.

	\begin{theorem}\label{thm:eigenfunctions}
		Let $N\geq 3$ and $x_0\in\Omega$. For every $\e\in(0,\e_0)$, let
		\begin{equation*}
			\mathcal{E}_\e^n:=\bigoplus_{i=1}^m E(\lambda_{n+i-1}(\Omega_\e^{x_0,\Sigma}))
		\end{equation*}
		and 
		\begin{equation*}
			\Pi_\e^n\colon L^2(\Omega_\e^{x_0,\Sigma})\to
			\mathcal{E}_\e^n\sub L^2(\Omega_\e^{x_0,\Sigma})
		\end{equation*}
		be the orthogonal projection. Then
		\begin{equation*}
			\norm{\varphi-\frac{\Pi_\e^n\varphi}
				{\norm{\Pi_\e^n\varphi}_{L^2(\Omega_\e^{x_0,\Sigma})}}}_{H^1(\Omega_\e^{x_0,\Sigma})}^2=
			\e^N\norm{\nabla
				U_{\Sigma,\nabla\varphi(x_0)\cdot\nu}}_{L^2(\R^N\setminus
				\Sigma)}^2+
			\mathcal R(\varphi,\e)
		\end{equation*}
		for all $\varphi\in E(\lambda_n(\Omega))$ such that
		$\norm{\varphi}_{L^2(\Omega)}=1$, where the reminder term
		$\mathcal R(\varphi,\e)$ is $o(\e^N)$ as $\e\to 0$ uniformly with
		respect to $\varphi$, i.e.
		\begin{equation*}
			\lim_{\e\to0}\sup_{\substack{\varphi\in E(\lambda_n(\Omega))\\
					\|\varphi\|_{L^2(\Omega)}=1}}\frac{|\mathcal R(\varphi,\e)|}{\e^N}=0.
		\end{equation*}
	\end{theorem}
	Having established the general framework, we now focus on the
	specific case $\Sigma = B_1$.  For every
	$N\geq 2$ and $x_0\in\Omega$, we consider the bilinear form
	\begin{align}\label{eq:defB}
		& \mathcal B_{x_0}: E(\lambda_n(\Omega))\times
		E(\lambda_n(\Omega))\to\R,\\
		\notag&\mathcal
		B_{x_0}(\varphi,\psi)=\omega_N\left(\frac{N}{N-1}\,\nabla\varphi(x_0)\cdot
		\nabla\psi(x_0)-(\lambda_n(\Omega)-1)\varphi(x_0)\psi(x_0)\right),
	\end{align}
	where $\omega_N=|B_1|$ is
	the $N$-dimensional Lebesgue measure of $B_1$.
	Let
	\begin{equation*}
		\{\gamma_i^{x_0}\}_{i=1,\dots,m}
	\end{equation*}
	denote the eigenvalues of $\mathcal
	B_{x_0}$ (in descending order).

	If $N\geq3$, we can provide an  explicit
	expression for the limit bilinear form \eqref{eq:limit-bi-form},
	proving  that this coincides with $\mathcal B_{x_0}$.

	\begin{proposition}\label{cor:ball_intr}
		Let $N\geq 3$, $x_0\in\Omega$, and $\mathcal{L}_{x_0,B_1}:
		E(\lambda_n(\Omega))\times E(\lambda_n(\Omega))\to\R$ be the
		bilinear form defined in \eqref{eq:limit-bi-form} with
		$\Sigma=B_1$. Then
		\begin{equation*}
			\mathcal{L}_{x_0,B_1}(\varphi,\psi)=\mathcal
			B_{x_0}(\varphi,\psi),\quad \text{for every $\varphi,\psi\in E(\lambda_n(\Omega))$}.
		\end{equation*}
	\end{proposition}
	If $N=2$, our strategy requires some adjustments. Indeed, the
	blow-up analysis for the torsional rigidity and the corresponding
	torsion functions, which works for $N\geq3$, is not available in the
	case $N=2$. Nevertheless, exploiting the study of torsion functions
	carried out in \cite[Section 6.2]{FLO}, in the case of spherical
	holes we are still able to obtain, even in planar domains, the same
	type of asymptotic expansion obtained above for $N\geq3$.
	\begin{theorem}\label{thm:exp-2d}
		Let $N=2$ and $x_0\in\Omega$.  Let
		$\{\gamma_i^{x_0}\}_{i=1,\dots,m}$ be the eigenvalues (in descending order) of the
		bilinear form \eqref{eq:defB}. 
		Then, for
		every $i=1,\dots,m$ we have
		\begin{equation*}
			\lambda_{n+i-1}\left(\Omega_\e^{x_0,B_1}\right)=
			\lambda_n(\Omega)-\gamma_i^{x_0}\e^2+o(\e^2)\quad\text{as }\e\to 0.
		\end{equation*}
	\end{theorem}
	Putting together the results of Theorem \ref{thm:1}, Proposition
	\ref{cor:ball_intr}, and Theorem \ref{thm:exp-2d}, we 
	derive directly the following corollary.
	\begin{corollary}\label{cor:spherical}
		Let $x_0\in\Omega$ and $N\geq2$.
		Then, for
		every $i=1,\dots,m$,
		\begin{equation*}
			\lambda_{n+i-1}\left(\Omega_\e^{x_0,B_1}\right)=
			\lambda_n(\Omega)-\gamma_i^{x_0}\e^N+o(\e^N)\quad\text{as }\e\to 0.
		\end{equation*}
	\end{corollary}
	This corollary is the final piece we need to obtain the proof of
	Theorem \ref{thm:3}, together with the characterization of the points
	$x_0$ for which eigenvalues $\gamma_i^{x_0}$ with different indexes
	overlap.
	
	\subsection{Notation}
	We collect here some notation that will be used throughout the
	paper:
	\begin{itemize}
		\item $\lambda_n:=\lambda_n(\Omega)$;
		\item we may drop the dependence on $x_0$ and $\Sigma$, and simply
		denote $\lambda_n^\e:=\lambda_n(\Omega_\e^{x_0,\Sigma})$;
		\item analogously, we may denote
		$\Omega_\e:=\Omega_\e^{x_0,\Sigma}$ and
		$\Sigma_\e:=\Sigma_\e^{x_0}$.
	\end{itemize}
	
	\subsection*{Outline of the paper} In the next section, we recall a
	notion of boundary torsional rigidity and prove some preliminary
	estimates on the torsion functions associated with eigenfunctions,
	aiming to obtain formulations that are uniform within the eigenspace.
	In section \ref{sec:asympt-mult-eigenv}, we apply the \emph{Lemma on
		small eigenvalues} by Y. Colin de Verdière \cite{ColindeV1986} to
	derive an asymptotic expansion of eigenbranches bifurcating from a
	multiple eigenvalue $\lambda_n(\Omega)$, consequently establishing
	Theorems \ref{thm:1}, \ref{thm:exp-2d}, and \ref{thm:eigenfunctions}.
	Section \ref{sec:splitting} is devoted to the proof of Theorem
	\ref{thm:3}, which is based on Corollary \ref{cor:spherical} and the
	remark that points $x_0$ where eigenvalues $\gamma_i^{x_0}$ overlap
	for different indexes $i$ correspond to the zeros of certain analytic
	functions.  Section \ref{sec:numerics} presents some numerical
	experiments in situations not covered by the theoretical results of
	this work, offering interesting directions for future investigation.
	Finally, \Cref{appendix} deals with the Dirichlet case, showing how
	our method recovers the known results established by Flucher in
	\cite{flucher}.

	\section{Preliminaries}
	
	\noindent
	We recall from \cite{FLO} the notion of boundary torsional rigidity
	suitable for dealing with the Neumann problem in a perforated domain,
	see also \cite{brasco2024}. Related notions of torsional rigidity
	have also arisen in other contexts, still within the framework of
	singularly perturbed eigenvalue problems, see
	e.g. \cite{AO2023,FNOS,O}.
	
	For any open Lipschitz set $E$ with $\overline{E}\sub \Omega$ and any
	$f\in L^2(\partial E)$, we defined the \emph{$f$-torsional rigidity of
		$\partial E$ relative to $\overline{\Omega}\setminus E$} as
	\begin{equation*}
		\mathcal{T}_{\overline{\Omega}\setminus E}(\partial E,f):=
		-2\min\left\{ \frac{1}{2}\int_{\Omega\setminus E}(|\nabla
		u|^2+u^2)\dx
		-\int_{\partial E}uf\ds\colon u\in H^1(\Omega\setminus
		\overline{E})
		\right\}.
	\end{equation*}
	We denote by $U_{\Omega,E,f}\in H^1(\Omega\setminus \overline{E})$ the unique
	minimizer. In the sequel, whenever there is no ambiguity, for any
	$\varphi\in E(\lambda_n)$ and $\e\in(0,\e_0)$, we simply denote
	\begin{equation}\label{eq:torfun}
		U_\e^\varphi:=U_{\Omega,\Sigma_\e,\partial_{\nnu}\varphi}.
	\end{equation}
	One can easily check that $U_\e^\varphi$ solves
	\begin{equation}\label{eq:equationUphi}
		\int_{\Omega_\e}(\nabla U_\e^\varphi\cdot\nabla u+U_\e^\varphi u)
		\dx=\int_{\partial\Sigma_\e}u\,\partial_{\nnu}\varphi\ds\quad\text{for
			all }
		u\in H^1(\Omega_\e).
	\end{equation}
	The following lemma provides an estimate of the vanishing order of the
	energy of the functions $U_\e^\varphi$ as $\e\to0$.
	\begin{lemma}\label{lemma:H1_decay}
		Let either $N=2$ and $\Sigma=B_1$ or $N\geq 3$. We have
		\begin{equation*}
			\sup_{\substack{\varphi\in E(\lambda_n) \\
					\norm{\varphi}_{L^2(\Omega)}=1 }}
			\norm{U_\e^\varphi}_{H^1(\Omega_\e)}^2=
			\sup_{\substack{\varphi\in E(\lambda_n) \\
					\norm{\varphi}_{L^2(\Omega)}=1 }}
			\int_{\Omega_\e}\big(|\nabla
			U_\e^\varphi|^2+|U_\e^\varphi|^2\big)\dx
			= O(\e^N)\quad\text{as }\e\to 0.
		\end{equation*}
	\end{lemma}
	\begin{proof}
		Let $\varphi\in E(\lambda_n)$ be such that
		$\norm{\varphi}_{L^2(\Omega)}=1$. We may write $\varphi$ as
		\begin{equation*}
			\varphi=\sum_{i=1}^m a_i\varphi_{n+i-1},
		\end{equation*}
		where $\sum_{i=1}^m a_i^2=1$. By linearity, we have
		$U_\e^\varphi=\sum_{i=1}^m a_i U_\e^{\varphi_{n+i-1}}$,
		hence,
		by the Cauchy-Schwarz inequality,
		\begin{align*}
			\norm{U_\e^\varphi}_{H^1(\Omega_\e)}^2&\leq 
			\left(\sum_{i=1}^m|a_i|\norm{U_\e^{\varphi_{n+i-1}}}_{H^1(\Omega_\e)}\right)^2\\
			&\leq
			\left(\sum_{i=1}^ma_i^2\right)\left(\sum_{i=1}^m
			\norm{U_\e^{\varphi_{n+i-1}}}_{H^1(\Omega_\e)}^2\right)
			\leq m
			\max_{i=1,\dots,m}\norm{U_\e^{\varphi_{n+i-1}}}_{H^1(\Omega_\e)}^2.
		\end{align*}
		In \cite{FLO} we proved that
		\begin{equation*}
			\norm{U_\e^{\varphi_{n+i-1}}}_{H^1(\Omega_\e)}^2=O(\e^{N+2k_i-2})
			\quad\text{as }\e\to 0,
		\end{equation*}
		where $k_i\geq 1$ is the vanishing order of
		$\varphi_{n+i-1}-\varphi_{n+i-1}(x_0)$ at $x_0$, with the
		exception of the case $N=2$ and $k_i=2$, for which
		$\norm{U_\e^{\varphi_{n+i-1}}}_{H^1(\Omega_\e)}^2=O(\e^{4}|\log\e|)$
		as $\e\to0$.  In particular, the case $N\geq 3$ is proved in
		\cite[Lemma 5.4]{FLO} (see also Remark 3.5), while the case $N=2$
		and $\Sigma=B_1$ can be found in \cite[Proposition 6.3]{FLO}. This
		concludes the proof.
	\end{proof}
	We also observe the following property of negligibility of the
	$L^2$ mass of functions $U_\e^\varphi$ with respect to their energy,
	uniformly in $E(\lambda_n)$ as $\e\to0$.
	\begin{lemma}\label{lemma:L2_decay}
		Let $N\geq 2$. Then
		\begin{equation}\label{eq:L2_decay_th1}
			\sup_{\substack{\varphi\in E(\lambda_n) \\
					\norm{\varphi}_{L^2(\Omega)}=1 }}
			\frac{\norm{U_\e^\varphi}_{L^2(\Omega_\e)}^2}
			{\norm{U_\e^\varphi}_{H^1(\Omega_\e)}^2}\to 0\quad\text{as }\e\to 0.
		\end{equation}
		In particular, letting
		\begin{equation*}
			\omega(\e):=\sup_{\substack{\varphi\in E(\lambda_n) \\
					\norm{\varphi}_{L^2(\Omega)}=1
			}}\norm{U_\e^\varphi}_{L^2(\Omega_\e)}^2,
		\end{equation*}
		if either $N=2$ and $\Sigma=B_1$ or $N\geq 3$, we have
		$\e^{-N}\omega(\e)\to 0$ as $\e\to 0$.
	\end{lemma}
	\begin{proof}
		In order to prove \eqref{eq:L2_decay_th1}, we proceed by
		contradiction. To this aim, we assume that there exist sequences
		$\{\e_i\}_i\subset(0,+\infty)$ and
		$\{\varphi^i\}_i\subset E(\lambda_n)$ such that
		$\lim_{i\to\infty}\e_i=0$, $\|\varphi^i\|_{L^2(\Omega)}=1$, and
		\begin{equation}\label{eq:L2_1}
			\frac{\|U_{\e_i}^{\varphi^i}\|_{L^2(\Omega_{\e_i})}^2}
			{\|U_{\e_i}^{\varphi^i}\|_{H^1(\Omega_{\e_i})}^2}\geq C
			\quad\text{for all }i\in\N,
		\end{equation}
		for some $C>0$. It is known (see e.g. \cite{SW99} and \cite{FSR})
		that, for every $\e\in(0,\e_0)$, there exists an extension
		operator $\mathsf{E}_\e\colon H^1(\Omega_\e)\to H^1(\Omega)$ such
		that
		\begin{equation*}
			\norm{\mathsf{E}_\e u}_{H^1(\Omega)}\leq
			\mathfrak{C}\norm{u}_{H^1(\Omega_\e)}\quad\text{for all }u\in
			H^1(\Omega_\e)\text{ and }\e>0,
		\end{equation*}
		where $\mathfrak{C}$ is a positive constant independent of $\e$.
		Therefore, letting
		\begin{equation*}
			U_i:=\frac{\mathsf{E}_{\e_i}U_{\e_i}^{\varphi^i}}
			{\|\mathsf{E}_{\e_i}U_{\e_i}^{\varphi^i}\|_{L^2(\Omega)}},
		\end{equation*}
		\eqref{eq:L2_1} implies that
		\begin{equation*}
			\norm{U_i}_{H^1(\Omega)}=O(1)\quad\text{as }i\to\infty.
		\end{equation*}
		Hence, up to extracting a subsequence, we have
		\begin{equation}\label{eq:convergences}
			U_i\weak U\quad\text{weakly in
			}H^1(\Omega)\quad\text{and}\quad
			U_i\to U\quad\text{strongly in }L^2(\Omega)\quad\text{as $i\to\infty$},
		\end{equation}
		for some $U\in H^1(\Omega)$ such that
		$\|U\|_{L^2(\Omega)}=1$.
		
		Let $v\in C^\infty(\overline{\Omega})$ be such that $v\equiv0$ in a
		neighborhood of $x_0$. Hence, for $i$
		sufficiently large, $v\equiv0$ in $\Sigma_{\e_i}$; thus, in
		view of the equation satisfied by $U_{\e_i}^{\varphi^i}$ we
		derive that
		\begin{equation*}
			\int_{\Omega}(\nabla U_i\cdot\nabla v+U_iv)\dx=
			\frac{1}{\|\mathsf{E}_{\e_i}U_{\e_i}^{\varphi^i}\|_{L^2(\Omega)}}
			\int_{\Omega_{\e_i}}(\nabla U_{\e_i}^{\varphi^i}
			\cdot\nabla v+U_{\e_i}^{\varphi^i}v)\dx=0.
		\end{equation*}
		In view of \eqref{eq:convergences}, we can pass to the limit as
		$i\to\infty$ in the above identity; then, since $\{x_0\}$ has
		null capacity, we can extend the validity to any test function
		$v\in H^1(\Omega)$, obtaining 
		\begin{equation*}
			\int_{\Omega}(\nabla U\cdot\nabla v+Uv)\dx=0
			\quad\text{for all }v\in H^1(\Omega).
		\end{equation*}
		This forces $U=0$; since $\|U\|_{L^2(\Omega)}=1$, this gives rise
		to a contradiction.
		
		The second part of the statement follows
		from \eqref{eq:L2_decay_th1} and \Cref{lemma:H1_decay}. 
	\end{proof}
	
	If $N\geq3$, the following stability estimate holds for the torsion
	functions  $U_\e^{\varphi}$.
	\begin{lemma}\label{l:stability}
		Let $N\geq3$. There exists a constant $S>0$, depending on
		$N,\Omega,\Sigma, \lambda_n$ (but independent of $\e$), such
		that, for all $\e\in(0,\e_0)$ and $\varphi\in E(\lambda_n)$,
		\begin{equation*}
			\|U_\e^{\varphi}\|_{H^1(\Omega_\e)}^2\leq
			S\left(\e\|\varphi\|_{L^2(\Sigma_\e)}
			+\|\nabla\varphi\|_{L^2(\Sigma_\e)}\right)^2.
		\end{equation*}
	\end{lemma}
	\begin{proof}
		The proof is contained in \cite{FLO}, see in particular estimate
		(5.7) in Lemma 5.4.
	\end{proof}

	For $N\geq3$, we can adapt the blow-up analysis carried out in
	\cite{FLO} to quantify the vanishing order of the energies of the
	torsion functions $U_\e^{\varphi}$.

	\begin{lemma}\label{lemma:blowup}
		Let $N\geq 3$. Then
		\begin{equation}\label{eq:sup}
			\rho_n(\e):=      \sup_{\substack{\varphi\in E(\lambda_n) \\
					\norm{\varphi}_{L^2(\Omega)}=1}}\abs{\e^{-N}
				\norm{U_\e^\varphi}_{H^1(\Omega_\e)}^2-\int_{\R^N\setminus\Sigma}
				|\nabla U_{\Sigma,\nabla\varphi(x_0)\cdot \nnu}|^2\dx}\to
			0\quad\text{as }\e\to 0.
		\end{equation}
	\end{lemma}
	\begin{proof}
		By continuity and compactness of the set
		$\{\varphi\in E(\lambda_n):\|\varphi\|_{L^2(\Omega)}=1\}$, for every
		$\e\in(0,\e_0)$ the supremum in \eqref{eq:sup} is achieved by some
		$\varphi_\e\in E(\lambda_n)$ with
		$\|\varphi_\e\|_{L^2(\Omega)}=1$. Since
		$\|\varphi_\e\|_{H^1(\Omega)}^2=\lambda_n$, for every sequence
		$\e_j\to0$, there exist a subsequence, still denoted as
		$\{\e_j\}_j$, and some $\varphi\in E(\lambda_n)$ such that
		$\|\varphi\|_{L^2(\Omega)}=1$ and $\varphi_{\e_j}\to \varphi$ in
		$L^2(\Omega)$ as $j\to+\infty$. By classical elliptic interior
		estimates, we deduce that $\varphi_{\e_j}\to \varphi$ in
		$C^1_{\rm loc}(\Omega)$.  Hence, by the linearity of
		$\varphi\mapsto U_\varepsilon^\varphi$, the Cauchy-Schwarz
		inequality, and Lemma \ref{l:stability} we have
		\begin{align}\label{eq:st1}
			&\left|\|U_{\e_j}^{\varphi_{\e_j}}\|_{H^1(\Omega_{\e_j})}^2-
			\|U_{\e_j}^{\varphi}\|_{H^1(\Omega_{\e_j})}^2\right|
			=\left|(U_{\e_j}^{\varphi_{\e_j}}-U_{\e_j}^{\varphi},
			U_{\e_j}^{\varphi_{\e_j}}+U_{\e_j}^{\varphi})_{H^1(\Omega_{\e_j})}\right|\\
			&\notag  =\left|(U_{\e_j}^{\varphi_{\e_j}-\varphi},
			U_{\e_j}^{\varphi_{\e_j}+\varphi})_{H^1(\Omega_{\e_j})}\right|
			\leq \|U_{\e_j}^{\varphi_{\e_j}-\varphi}\|_{H^1(\Omega_{\e_j})}
			\|U_{\e_j}^{\varphi_{\e_j}+\varphi}\|_{H^1(\Omega_{\e_j})}\\
			&\notag\leq
			S\left(\e_j\|\varphi_{\e_j}-\varphi\|_{L^2(\Sigma_\e)}+
			\|\nabla(\varphi_{\e_j}-\varphi)\|_{L^2(\Sigma_\e)}\right)
			\left(\e_j\|\varphi_{\e_j}+\varphi\|_{L^2(\Sigma_\e)}+\|\nabla(\varphi_{\e_j}+\varphi)\|_{L^2(\Sigma_\e)}\right)\\
			&\notag=o(\e_j^N)\quad\text{as }j\to+\infty.
		\end{align}
		In a similar way,
		\begin{multline}\label{eq:st2}
			\left|  \|\nabla U_{\Sigma,\nabla\varphi_{\e_j}(x_0)\cdot
				\nnu}\|_{L^2(\R^N\setminus\Sigma)}^2-
			\|\nabla U_{\Sigma,\nabla\varphi(x_0)\cdot
				\nnu}\|_{L^2(\R^N\setminus\Sigma)}^2\right|\\\leq
			\|\nabla U_{\Sigma,\nabla(\varphi_{\e_j}-\varphi)(x_0)\cdot
				\nnu}\|_{L^2(\R^N\setminus\Sigma)}
			\|\nabla U_{\Sigma,\nabla(\varphi_{\e_j}+\varphi)(x_0)\cdot
				\nnu}\|_{L^2(\R^N\setminus\Sigma)}
			=o(1)\quad\text{as }j\to+\infty.
		\end{multline}
		Furthermore, from  \cite[Theorem 5.6]{FLO} it is known that 
		\begin{equation}\label{eq:st3}
			\e^{-N} \norm{U_\e^{\varphi}}_{H^1(\Omega_\e)}^2\to
			\|\nabla U_{\Sigma,\nabla\varphi(x_0)\cdot
				\nnu}\|_{L^2(\R^N\setminus\Sigma)}^2\quad\text{as }\e\to 0.
		\end{equation}
		From \eqref{eq:st1}, \eqref{eq:st2}, and \eqref{eq:st3} it follows that
		\begin{align*}
			\rho_n(\e_j)&=  \bigg|\e_j^{-N}
			\left(\|U_{\e_j}^{\varphi_{\e_j}}\|_{H^1(\Omega_{\e_j})}^2-
			\|U_{\e_j}^{\varphi}\|_{H^1(\Omega_{\e_j})}^2\right)\\
			&\quad-
			\left(\|\nabla U_{\Sigma,\nabla\varphi_{\e_j}(x_0)\cdot
				\nnu}\|_{L^2(\R^N\setminus\Sigma)}^2-
			\|\nabla U_{\Sigma,\nabla\varphi(x_0)\cdot
				\nnu}\|_{L^2(\R^N\setminus\Sigma)}^2\right)\\
			&\quad +\bigg(\e_j^{-N} \|U_{\e_j}^\varphi\|_{H^1(\Omega_{\e_j})}^2-
			\|\nabla U_{\Sigma,\nabla\varphi(x_0)\cdot
				\nnu}\|_{L^2(\R^N\setminus\Sigma)}^2\bigg)\bigg|
			=o(1)\quad\text{as }j\to+\infty.
		\end{align*}
		The conclusion then follows by Urysohn's subsequence principle.
	\end{proof}
	
	\section{Asymptotics of multiple eigenvalues}\label{sec:asympt-mult-eigenv}
	
	\noindent
	In the whole section, we consider a fixed eigenvalue
	$\lambda_n=\lambda_n(\Omega)$ of problem \eqref{eq:unperturbed} with
	multiplicity $m:=\dim E(\lambda_n)$, where the index $n\in\N$ is
	chosen as in \eqref{eq:scelta-n}.
	
	For every $\e\in(0,\e_0)$, we
	introduce the bilinear form
	\begin{align}\label{eq:choices1}
		&q_\e: H^1(\Omega_\e)\times H^1(\Omega_\e)\to\R,\\
		&\notag 
		q_\e (u,v)
		:=\int_{\Omega_\e}(\nabla u\cdot\nabla
		v+uv)\dx-\lambda_n\int_{\Omega_\e}uv\dx. 
	\end{align}
	For every $\varphi\in E(\lambda_n)$ and $\e\in(0,\e_0)$, we denote
	\begin{equation*}
		U_\e^\varphi:=U_{\Omega,\Sigma_\e,\partial_{\nnu}\varphi}, \quad
		P_\e(\varphi):=\varphi-U_\e^\varphi,
	\end{equation*}
	and
	\begin{equation}\label{eq:choices2}
		F_\e:=\{P_\e(\varphi)\colon\varphi\in E(\lambda_n)\}
		=P_\e(E(\lambda_n))\sub L^2(\Omega_\e).
	\end{equation}
	Treating $E(\lambda_n)$ as a subspace of
	$L^2(\Omega)$ and defining $\omega(\e)$ as in Lemma
	\ref{lemma:L2_decay}, for every $\e\in(0,\e_0)$
	we have
	\begin{equation*}
		\norm{I-P_\e}_{\mathcal{L}(E(\lambda_n);L^2(\Omega_\e))}^2=\omega(\e),
	\end{equation*}
	where $I$ is the identity operator. If  either $N\geq 2$ and
	$\Sigma=B_1$ or $N\geq 3$, by Lemma \ref{lemma:L2_decay}
	$\lim_{\e\to0}\omega(\e)=0$, hence $P_\e:E(\lambda_n)\to F_\e$
	is a bijection for $\e$ sufficiently small and $F_\e$ is
	$m$-dimensional.
	
	The following \emph{Lemma on small eigenvalues} by Y. Colin de
	Verdière \cite{ColindeV1986} (see also \cite[Proposition B.1]{ALM2022})
	is a crucial tool in the proof of Theorem \ref{thm:1}.
	\begin{lemma}[Lemma on small eigenvalues \cite{ColindeV1986}]\label{lemma:CdV}
		Let $(\mathcal{H},(\cdot,\cdot)_{\mathcal{H}})$ be a real
		Hilbert space and let $\mathcal{D}\sub\mathcal{H}$ be a dense
		subspace. Let $q\colon \mathcal{D}\times\mathcal{D}\to\R$ be a
		symmetric bilinear form, such that $q$ is semi-bounded from below,
		$q$ admits an increasing sequence of eigenvalues
		$\{\nu_i\}_{i\in\N}$, 
		and there exists an orthonormal basis $\{g_i\}_{i\in\N}$ of
		$\mathcal H$ such that $g_i\in\mathcal D$ is an eigenvector of $q$
		associated to the eigenvalue $\nu_i$, i.e.  $q(g_i,v)=\nu_i(g_i,v)_{\mathcal{H}}$
		for all $i\geq1$ and $v\in \mathcal D$.
		
		Let $m\in \N\setminus\{0\}$ and $F\sub \mathcal{D}$ be
		a $m$-dimensional subspace of $\mathcal D$.
		
		Assume that there exist $n\in\N$ and
		$\gamma>0$ such that
		\begin{itemize}
			\item[\rm (H1)] $\nu_i\leq -\gamma$ for all $i\leq n-1$,
			$|\nu_i|\leq \gamma$ for all $i=n,\dots,n+m-1$, and 
			$\nu_i\geq \gamma$ for all $i\geq n+m$;
			\item[\rm (H2)]
			$\delta:=\sup\{|q(u,v)|\colon u\in \mathcal{D},~v\in
			F,~\norm{u}_{\mathcal{H}}=\norm{v}_{\mathcal{H}}=1\}<\gamma/\sqrt{2}$.
		\end{itemize}
		If $\{\xi_i\}_{i=1,\dots,m}$ are the eigenvalues (in
		ascending order) of $q$ restricted to $F$, we have
		\begin{equation}\label{eq:CdV_th1}
			|\nu_{n+i-1}-\xi_i|\leq\frac{4\delta^2}{\gamma}
			\quad\text{for all }i=1,\dots,m.
		\end{equation}
		Moreover, if
		$\Pi\colon\mathcal{D}\to
		\mathrm{span}\,\{g_n,\dots,g_{n+m-1}\}$ denotes the
		orthogonal projection onto the subspace of
		$\mathcal D$  spanned by $\{g_n,\dots,g_{n+m-1}\}$,
		we have
		\begin{equation*}
			\frac{\norm{v-\Pi
					v}_{\mathcal{H}}}{\norm{v}_{\mathcal{H}}}
			\leq \frac{\sqrt{2}\delta}{\gamma}\quad\text{for all }v\in F.
		\end{equation*}
	\end{lemma}
	
	We are going to apply \Cref{lemma:CdV} in our setting, in order to
	prove the following expansion for the perturbed eigenvalues.
	\begin{proposition}\label{prop:CdV}
		Let either $N\geq 2$ and $\Sigma=B_1$ or $N\geq 3$. Let
		$\{\xi_i^\e\}_{i=1,\dots,m}$ be the eigenvalues of the bilinear form
		$q_\e$ defined in \eqref{eq:choices1}, restricted
		to the space $F_\e$ introduced in \eqref{eq:choices2}. Then
		\begin{equation}\label{eq:appl_CdV_th1}
			\lambda_{n+i-1}^\e=\lambda_n+\xi_i^\e+o(\e^N)\quad\text{as }\e\to 0,
		\end{equation}
		for all $i=1,\dots,m$. Moreover,
		\begin{equation}\label{eq:appl_CdV_th2}
			\sup_{\substack{\varphi\in
					E(\lambda_n)\\\|\varphi\|_{L^2(\Omega)}=1}}
			\norm{P_\e(\varphi)-\Pi_\e^n P_\e(\varphi)}_{L^2(\Omega_\e)}^2
			=o(\e^N)\quad\text{as }\e\to 0,
		\end{equation}
		where $\Pi_\e^n$ is as in \Cref{thm:eigenfunctions}.
	\end{proposition}
	\begin{proof}
		In order  to apply \Cref{lemma:CdV} in our setting, we choose:
		\begin{equation*}
			\mathcal{H}:=L^2(\Omega_\e),\quad\mathcal{D}:=H^1(\Omega_\e),
			\quad q:=q_\e,\quad 
			F:=F_\e,
		\end{equation*}
		being $q_\e$ and $F_\e$ as in \eqref{eq:choices1} and
		\eqref{eq:choices2} respectively, so that
		\begin{equation*}
			\nu_i=\lambda_i^\e-\lambda_n\quad\text{for all }i\in\N.
		\end{equation*}
		Let us verify that the hypotheses of Lemma \ref{lemma:CdV}
		are satisfied with the choices made above.  By
		\eqref{eq:conveig} $\lambda_i^\e\to\lambda_i$ as $\e\to 0$, for
		all $i\in\N$; hence
		\begin{align*}
			&\nu_i=\lambda_i^\e-\lambda_n\leq
			\frac{1}{2}(\lambda_{n-1}-\lambda_n)<0\quad\text{for every }i<n, \\
			&\nu_i=\lambda_i^\e-\lambda_n\geq
			\frac{1}{2}(\lambda_{n+m}-\lambda_{n+m-1})>0
			\quad\text{for every }i>n+m-1,
		\end{align*}
		for $\e$ sufficiently small (depending on $n$). Therefore,
		assumption (H1) of Lemma \ref{lemma:CdV} is satisfied with
		\begin{equation}\label{eq:gamma}
			\gamma:=\frac{1}{2}\min\{\lambda_n-\lambda_{n-1},
			\lambda_{n+m}-\lambda_n\}>0
		\end{equation}
		and $\e$ sufficiently small.

		In order to estimate the quantity $\delta$ in assumption (H2)
		of Lemma \ref{lemma:CdV}, we first fix $u\in H^1(\Omega_\e)$
		and $v\in F_\e$ and compute $q_\e(u,v)$. Being $P_\e$ a bijection,
		there exists a unique $\varphi\in E(\lambda_n)$ such
		that $v=\varphi-U_\e^\varphi$, so that
		\begin{equation*}
			q_\e(u,v)=q_\e(u,\varphi-U_\e^\varphi)=\int_{\Omega_\e}
			(\nabla u\cdot\nabla
			(\varphi-U_\e^\varphi)+u(\varphi-U_\e^\varphi))\dx-
			\lambda_n\int_{\Omega_\e}u(\varphi-U_\e^\varphi)\dx.
		\end{equation*}
		By integrating by parts and taking into account the equation
		satisfied by $U_\e^\varphi$, i.e. \eqref{eq:equationUphi}, we
		obtain that
		\begin{equation*}
			q_\e(u,\varphi-U_\e^\varphi)=\lambda_n\int_{\Omega_\e}u U_\e^\varphi\dx,
		\end{equation*}
		so that, by the Cauchy-Schwarz inequality,
		\begin{equation}\label{eq:CdV_1}
			\delta\leq \lambda_n\sqrt{\omega(\e)},
		\end{equation}
		where $\omega(\e)$ is as in \Cref{lemma:L2_decay}. Hence, from
		\Cref{lemma:L2_decay} itself, (H2) is satisfied for $\e$
		sufficiently small, and expansion \eqref{eq:appl_CdV_th1} follows
		from \eqref{eq:CdV_th1}.
		
		In order to prove \eqref{eq:appl_CdV_th2}, we first observe that,
		by \Cref{lemma:CdV} and \eqref{eq:CdV_1}, for every
		$\varphi\in E(\lambda_n)$ such that
		$\norm{\varphi}_{L^2(\Omega)}=1$ we have
		\begin{equation}\label{eq:prstpe}
			\norm{P_\e(\varphi)-\Pi_\e^n
				P_\e(\varphi)}_{L^2(\Omega_\e)}^2\leq
			\frac{2\lambda_n^2}{\gamma^2}\omega(\e)
			\|P_\e(\varphi)\|_{L^2(\Omega_\e)}^2,
		\end{equation}
		with $\gamma$ and $\omega(\e)$ being as in
		\eqref{eq:gamma} and \Cref{lemma:L2_decay} respectively. 
		Let us now analyze the rightmost
		term: we have
		\begin{align}\label{eq:CdV_2}
			\norm{P_\e(\varphi)}_{L^2(\Omega_\e)}^2&\leq 
			2\norm{\varphi}_{L^2(\Omega_\e)}^2+2\norm{(P_\e-I)(\varphi)}_{L^2(\Omega_\e)}^2\\
			\notag &\leq  2\norm{\varphi}_{L^2(\Omega)}^2+2\omega(\e)\|\varphi\|_{L^2(\Omega)}^2 =2(1+\omega(\e))
		\end{align}
		for all $\varphi\in E(\lambda_n)$ with
		$\|\varphi\|_{L^2(\Omega)}=1$. The conclusion follows from
		\eqref{eq:prstpe}, \eqref{eq:CdV_2}, and \Cref{lemma:L2_decay}.
	\end{proof}
	
	In order to characterize the eigenvalues $\{\xi_i^\e\}_{i=1,\dots,m}$ that come
	into play in Proposition \ref{prop:CdV}, we introduce the bilinear
	form $r_\e\colon E(\lambda_n)\times E(\lambda_n) \to \R$ defined as
	\begin{equation}\label{eq:def-r-eps}
		r_\e(\varphi,\psi):=q_\e(P_\e(\varphi),P_\e(\psi)).
	\end{equation}
	\begin{lemma}\label{lemma:r_e}
		We have 
		\begin{align*}
			r_\e(\varphi,\psi)=
			&-\int_{\Sigma_\e}(\nabla\varphi\cdot\nabla\psi-(\lambda_n-1)
			\varphi\psi)\dx\\
			&-\int_{\Omega_\e}(\nabla U_\e^\varphi\cdot\nabla
			U_\e^\psi
			+U_\e^\varphi U_\e^\psi)\dx-\lambda_n\int_{\Omega_\e}U_\e^\varphi U_\e^\psi\dx
		\end{align*}
		for all $\varphi,\psi\in E(\lambda_n)$.
	\end{lemma}
	\begin{proof}
		For every $\varphi,\psi\in E(\lambda_n)$ we have
		\begin{align*}
			r_\e(\varphi,\psi)&=q_\e(\varphi-U_\e^\varphi,\psi-U_\e^\psi)\\
			&=(\varphi,\psi)_{H^1(\Omega_\e)}-\lambda_n(\varphi,\psi)_{L^2(\Omega_\e)} 
			-\Big( (\varphi,U_\e^\psi)_{H^1(\Omega_\e)}-\lambda_n(\varphi,U_\e^\psi)_{L^2(\Omega_\e)} \Big) \\
			&\quad -\Big(
			(\psi,U_\e^\varphi)_{H^1(\Omega_\e)}-\lambda_n(\psi,U_\e^\varphi)_{L^2(\Omega_\e)}
			\Big)
			+(U_\e^\varphi,U_\e^\psi)_{H^1(\Omega_\e)}-\lambda_n(U_\e^\varphi,U_\e^\psi)_{L^2(\Omega_\e)}.
		\end{align*}
		Since $\varphi\in E(\lambda_n)$ we have
		\begin{equation*}
			(\varphi,\psi)_{H^1(\Omega_\e)}-\lambda_n(\varphi,\psi)_{L^2(\Omega_\e)}
			=-\int_{\Sigma_\e}(\nabla\varphi\cdot\nabla\psi-(\lambda_n-1)\varphi\psi)\dx.
		\end{equation*}
		On the other hand, by integrating by parts and using the equations
		satisfied by $\varphi$ and $U_\e^\varphi$, respectively
		$\psi$ and $U_\e^\psi$, one can prove that
		\begin{equation*}
			(\varphi,U_\e^\psi)_{H^1(\Omega_\e)}-\lambda_n(\varphi,U_\e^\psi)_{L^2(\Omega_\e)}
			=(U_\e^\varphi,U_\e^\psi)_{H^1(\Omega_\e)},
		\end{equation*}
		respectively
		\begin{equation*}
			(\psi,U_\e^\varphi)_{H^1(\Omega_\e)}-\lambda_n(\psi,U_\e^\varphi)_{L^2(\Omega_\e)}
			=(U_\e^\varphi,U_\e^\psi)_{H^1(\Omega_\e)}.
		\end{equation*}
		Summing up all the terms, we conclude the proof.
	\end{proof}
	
	\begin{lemma}\label{lemma:xi_mu}
		Let either $N=2$ and $\Sigma=B_1$ or $N\geq 3$. Let
		$\{\xi_i^\e\}_{i=1,\dots,m}$ be the eigenvalues of $q_\e$ restricted
		to $F_\e$ and let $\{\mu_i^\e\}_{i=1,\dots,m}$ be the eigenvalues of
		$r_\e$. Then
		\begin{equation}\label{eq:exp-mu-xi}
			\xi_i^\e=\mu_i^\e+o(\e^N)\quad\text{as }\e\to 0,
		\end{equation}
		for all $i=1,\dots,m$.
	\end{lemma}
	\begin{proof}
		Let $\{\varphi_k\}_{k\in\N}$ be as in the introduction,
		i.e. $\{\varphi_k\}_{k\in\N}$ is an orthonormal basis of $L^2(\Omega)$ with
		$\varphi_k$ being an eigenfunction of problem \eqref{eq:unperturbed}
		associated to the eigenvalue $\lambda_k(\Omega)$. Hence, by
		\eqref{eq:scelta-n}, $\{\varphi_{n+i-1}\}_{i=1}^m$ is an orthonormal
		basis of $E(\lambda_n)$.
		
		Therefore, the eigenvalues $\{\mu_i^\e\}_{i=1,\dots,m}$ of $r_\e$
		coincide with the eigenvalues of the $m\times m$ symmetric real matrix
		\begin{equation*}
			R_\e=\big(r_\e(\varphi_{n+i-1}, \varphi_{n+j-1})\big)_{ij}.
		\end{equation*}
		We observe that the bilinear form $q_\e$ restricted to $F_\e$ coincides
		with the bilinear form
		\begin{equation*}
			(u,v)\in  F_\e\times F_\e \ \mapsto\  r_\e\big( P_\e^{-1}(u),
			P_\e^{-1}(v)\big).
		\end{equation*}
		Hence the eigenvalues $\{\xi_i^\e\}_{i=1,\dots,m}$ of $q_\e$
		restricted to $F_\e$ coincide with the eigenvalues of the
		$m\times m$ symmetric real matrix
		\begin{equation}\label{eq:xi_mu_1}
			B_\e=C_\e^{-1}R_\e,
		\end{equation}
		where
		\begin{equation*}
			C_\e=\big((P_\e\varphi_{n+i-1}, P_\e\varphi_{n+j-1})_{L^2(\Omega_\e)}\big)_{ij}.
		\end{equation*}   
		Since
		\begin{equation*}
			(C_\e)_{ij}=\int_{\Omega_\e}(\varphi_{n+i-1}-U_\e^{\varphi_{n+i-1}})
			(\varphi_{n+j-1}-U_\e^{\varphi_{n+j-1}})\dx\quad\text{for all }i,j=1,\dots,m,
		\end{equation*}
		from the Cauchy-Schwarz inequality and \Cref{lemma:L2_decay} it
		follows that
		\begin{align*}
			(C_\e)_{ii}&=1+o(\e^{\frac{N}{2}})\quad\text{for all }i=1,\dots,m, \\
			(C_\e)_{ij}&=o(\e^{\frac{N}{2}})\quad\text{for all }i,j=1,\dots,m,~i\neq j,
		\end{align*}
		as $\e\to 0$. By basic linear algebra, we then obtain that
		\begin{align*}
			(C_\e^{-1})_{ii}&=1+o(\e^{\frac{N}{2}})\quad\text{for }i=1,\dots,m, \\
			(C_\e^{-1})_{ij}&=o(\e^{\frac{N}{2}})\quad\text{for }i,j=1,\dots,m,~i\neq j,
		\end{align*}
		as $\e\to 0$. Moreover, from \Cref{lemma:r_e} and
		\Cref{lemma:H1_decay} we deduce that $(R_\e)_{ij}=O(\e^N)$ as
		$\e\to 0$, for all $i,j=1,\dots,m$. Combining these estimates, in
		view of \eqref{eq:xi_mu_1} we have
		\begin{equation}\label{eq:expBeRe}
			B_\e=R_\e+o(\e^N)\quad \text{as }\e\to0,
		\end{equation}
		where $o(\e^N)$ stands for a matrix with all
		entries being $o(\e^N)$ as $\e\to0$.
		The expansion \eqref{eq:exp-mu-xi} then follows from \eqref{eq:expBeRe} and
		the min-max characterization of eigenvalues.
	\end{proof}

	\begin{lemma}\label{lemma:r_e_blowup}
		Let $r_\e$ be the biliner form defined in \eqref{eq:def-r-eps}.
		\begin{enumerate}[label = (\roman*)]
			\item If $N\geq 3$, then 
			\begin{equation*}
				r_\e(\varphi,\psi)=-\e^N
				\mathcal{L}_{x_0,\Sigma}(\varphi,\psi)+o(\e^N)
				\quad \text{as }\e\to0,
			\end{equation*}
			for all $\varphi,\psi\in E(\lambda_n)$, where
			$\mathcal{L}_{x_0,\Sigma}$ is defined in \eqref{eq:limit-bi-form}.
			\item If $N=2$ and $\Sigma=B_1$, then
			\begin{equation*}
				r_\e(\varphi,\psi)=-\e^2\pi\Big(2\nabla\varphi(x_0)\cdot
				\nabla\psi(x_0)
				-(\lambda_n-1)\varphi(x_0) \psi(x_0)\Big)+o(\e^2)
				\quad \text{as }\e\to0,
			\end{equation*}    
			for all $\varphi,\psi\in E(\lambda_n)$.
		\end{enumerate}
	\end{lemma}
	\begin{proof}
		If $N\geq3$, by Lemma \ref{lemma:r_e} and \cite[Theorem 5.6 and Lemma 3.6]{FLO} (see also Lemma
		\ref{lemma:blowup} and \eqref{eq:L2_decay_th1}), we have
		\begin{equation*}
			\lim_{\e\to0}\e^{-N}r_\e(u,u)=- \mathcal{L}_{x_0,\Sigma}(u,u)
			\quad\text{for all }u\in E(\lambda_n).
		\end{equation*}
		Statement (i) then follows simply observing  that, by bilinearity,
		\begin{align}
			\notag r_\e(\varphi,\psi)&=\frac14\Big(r_\e(\varphi+\psi,
			\varphi+\psi)-r_\e(\varphi-\psi, \varphi-\psi)\Big),\\
			\label{eq:bilform} \mathcal{L}_{x_0,\Sigma} (\varphi,\psi)&=\frac14 \Big ( \mathcal{L}_{x_0,\Sigma} (\varphi+\psi,
			\varphi+\psi)- \mathcal{L}_{x_0,\Sigma} (\varphi-\psi, \varphi-\psi) \Big),
		\end{align}
		for all $\varphi,\psi\in E(\lambda_n)$.
		
		If $N=2$ and $\Sigma=B_1$, from Lemma \ref{lemma:r_e},
		\cite[Proposition 6.3]{FLO}, and \cite[Lemma 3.6]{FLO} (see also
		\eqref{eq:L2_decay_th1}) it follows that
		\begin{equation*}
			\lim_{\e\to0}\e^{-2}r_\e(u,u)=-
			\pi\left(2\,|\nabla u(x_0)|^2-
			(\lambda_n(\Omega)-1)u^2(x_0)\right)
			\quad\text{for all }u\in E(\lambda_n).
		\end{equation*}
		Statement (ii) then follows by bilinearity as above.
	\end{proof}

	At this point, the proof of Theorems \ref{thm:1} and \ref{thm:exp-2d} are straightforward.

	\begin{proof}[Proof of Theorems \ref{thm:1}  and \ref{thm:exp-2d}]
		The proof follows from a combination of Proposition \ref{prop:CdV},
		Lemma \ref{lemma:xi_mu} and \Cref{lemma:r_e_blowup}.
	\end{proof}
	
	We are now in position to prove \Cref{cor:ball_intr}.
	
	\begin{proof}[Proof of \Cref{cor:ball_intr}]
		First of all, we recall the explicit expression of the limit torsion
		function $U_{B_1,\nabla\varphi(x_0)\cdot\nnu}$ derived in the
		proof of \cite[Lemma 6.1]{FLO}
		in the case of a spherical hole. 
		More precisely, for all $\bm{b}\in\R^N$ we have 
		\begin{equation*}
			U_{B_1,\bm{b}\cdot\nnu}(x)=-\frac{1}{N-1}|x|^{-N}\bm{b}\cdot x.
		\end{equation*}
		Therefore,
		\begin{equation*}
			\tau_{\R^N\setminus B_1}(\partial B_1, \bm{b}\cdot\nnu)=\int_{\R^N\setminus
				B_1}
			|\nabla U_{B_1,\bm{b}\cdot\nnu}|^2\dx=\frac{\omega_N}{N-1}|\bm{b}|^2,
		\end{equation*}
		so that 
		\begin{equation*}
			\mathcal{L}_{x_0,B_1}(\varphi,
			\varphi)=\omega_N\left(\frac{N}{N-1}|\nabla\varphi(x_0)|^2
			-(\lambda_n-1)\varphi^2(x_0)\right), \quad\text{for all }\varphi\in E(\lambda_n).
		\end{equation*}
		The conclusion therefore follows from \eqref{eq:bilform}.  
	\end{proof}
	
	We conclude this section with the proof of Theorem
	\ref{thm:eigenfunctions}, concerning the asymptotic behavior of the
	energy of the difference between eigenfunctions and their projection
	onto the perturbed eigenspaces.
	
	\begin{proof}[Proof of \Cref{thm:eigenfunctions}]
		Let us fix $\varphi\in
		E(\lambda_n)$ such that
		$\norm{\varphi}_{L^2(\Omega)}=1$. First of all, we observe that
		$f_\e^\varphi:=P_\e(\varphi)-\Pi_\e^nP_\e(\varphi)$ weakly solves
		\begin{equation*}
			\begin{bvp}
				-\Delta f_\e^\varphi +f_\e^\varphi &= \lambda_n
				f_\e^\varphi+\lambda_n\Pi_\e^nP_\e(\varphi)+\lambda_nU_\e^\varphi
				+\Delta(\Pi_\e^nP_\e(\varphi))-\Pi_\e^nP_\e(\varphi), &&\text{in }\Omega_\e, \\
				\partial_{\nnu}f_\e^\varphi&=0, &&\text{on
				}\partial\Omega_\e.
			\end{bvp}
		\end{equation*}
		If we multiply both sides by
		$f_\e^\varphi$ and integrate by parts on
		$\Omega_\e$, we
		obtain
		\begin{multline}\label{eq:eigenfunctions_1}
			\norm{f_\e^\varphi}_{H^1(\Omega_\e)}^2=\lambda_n\norm{f_\e^\varphi}_{L^2(\Omega_\e)}^2 \\
			+\lambda_n\int_{\Omega_\e}U_\e^\varphi
			f_\e^\varphi\dx+\int_{\Omega_\e}(\Delta(\Pi_\e^nP_\e(\varphi))
			-\Pi_\e^nP_\e(\varphi)+\lambda_n\Pi_\e^nP_\e(\varphi))f_\e^\varphi\dx.
		\end{multline}
		The rest of the proof is divided into four steps.
		
		\smallskip\noindent\textbf{Step 1:} we claim that 
		\begin{equation}\label{eq:eigenfunctions_2}
			\sup_{\substack{\varphi\in
					E(\lambda_n)\\\|\varphi\|_{L^2(\Omega)}=1}}\norm{f_\e^\varphi}_{H^1(\Omega_\e)}^2=
			o(\e^N)\quad\text{as }\e\to 0.
		\end{equation}
		In order to derive \eqref{eq:eigenfunctions_2}, we prove that each
		term on the right hand side of \eqref{eq:eigenfunctions_1} is
		$o(\e^N)$ as $\e\to 0$, uniformly with respect to $\varphi$. One
		can immediately observe that this is true for the first term and
		(by the Cauchy-Schwarz inequality) for the second term, in view of
		\eqref{eq:appl_CdV_th2} and \Cref{lemma:L2_decay}. For what
		concerns the last term in \eqref{eq:eigenfunctions_1}, we let
		$\{\varphi_{n+i-1}^\e\}_{i=1,\dots,m}$ be a
		$L^2(\Omega_\e)$-orthonormal eigenbasis of $\mathcal{E}_\e^n$, so
		that we can write
		\begin{equation*}
			\Pi_\e^nP_\e(\varphi)=\sum_{i=1}^m a_i^\e\varphi_{n+i-1}^\e,
			\quad\text{with}\quad\sum_{i=1}^m |a_i^\e|^2=
			\norm{\Pi_\e^nP_\e(\varphi)}_{L^2(\Omega_\e)}^2.
		\end{equation*}
		We have 
		\begin{equation*}
			-\Delta(\Pi_\e^nP_\e(\varphi))+\Pi_\e^nP_\e(\varphi)-
			\lambda_n\Pi_\e^nP_\e(\varphi)=\sum_{i=1}^m
			a_i^\e(\lambda_{n+i-1}^\e-\lambda_n)\varphi_{n+i-1}^\e
		\end{equation*}
		and hence
		\begin{align*}
			\|-\Delta(\Pi_\e^nP_\e(\varphi))&+\Pi_\e^n
			P_\e(\varphi)-\lambda_n\Pi_\e^nP_\e(\varphi)\|_{L^2(\Omega_\e)}^2
			=\sum_{i=1}^m|a_i^\e|^2|\lambda_{n+i-1}^\e-\lambda_n|^2  \\
			&\leq \Big(\max_{i=1,\dots,m}|\lambda_{n+i-1}^\e-\lambda_n|^2\Big)
			\norm{\Pi_\e^nP_\e(\varphi)}_{L^2(\Omega_\e)}^2 \\
			&\leq \Big(\max_{i=1,\dots,m}|\lambda_{n+i-1}^\e-\lambda_n|^2 \Big)
			\norm{P_\e(\varphi)}_{L^2(\Omega_\e)}^2 ,
		\end{align*}
		where we also used that $\Pi_\e^n$ is an orthogonal projection. On
		the other hand,
		by \eqref{eq:CdV_2} and Lemma \ref{lemma:L2_decay} we have 
		\begin{equation*}
			\sup_{\substack{\varphi\in
					E(\lambda_n)\\\|\varphi\|_{L^2(\Omega)}=1}}
			\norm{P_\e(\varphi)}_{L^2(\Omega_\e)}^2\leq
			2(1+\omega(\e))=O(1)\quad\text{as }\e\to0.
		\end{equation*}
		Moreover, \Cref{thm:1} implies that
		\begin{equation*}
			\max_{i=1,\dots,m}|\lambda_{n+i-1}^\e-\lambda_n|
			=O(\e^N)\quad\text{as }\e\to 0.
		\end{equation*}
		Applying the Cauchy-Schwarz inequality to the last term
		in \eqref{eq:eigenfunctions_1} and combining the above estimates 
		with \eqref{eq:appl_CdV_th2}, we obtain
		\begin{equation*}
			\abs{\int_{\Omega_\e}(\Delta(\Pi_\e^nP_\e(\varphi))-
				\Pi_\e^nP_\e(\varphi)+\lambda_n\Pi_\e^nP_\e(\varphi))
				f_\e^\varphi\dx}=o(\e^N)\quad\text{as }\e\to 0,
		\end{equation*}
		uniformly with respect to $\varphi$. This completes the proof of
		claim \eqref{eq:eigenfunctions_2}.
		
		\smallskip    \noindent\textbf{Step 2:} we claim that 
		\begin{equation}\label{eq:eigenfunctions_3}
			\sup_{\substack{\varphi\in
					E(\lambda_n)\\\|\varphi\|_{L^2(\Omega)}=1}}
			\norm{P_\e(\varphi)-\Pi_\e^n\varphi}_{H^1(\Omega_\e)}^2=
			o(\e^N)\quad\text{as }\e\to0.
		\end{equation}
		We can write $\Pi_\e^n P_\e(\varphi)$ and $\Pi_\e^n \varphi$ as 
		\begin{equation*}
			\Pi_\e^n P_\e(\varphi)=\sum_{i=1}^m a_i^\e \varphi_{n+i-1}^\e
			\quad\text{and}\quad \Pi_\e^n\varphi=\sum_{i=1}^m b_i^\e
			\varphi_{n+i-1}^\e,
		\end{equation*}
		so that 
		\begin{align*}
			\norm{\Pi_\e^n
				P_\e(\varphi)-\Pi_\e^n\varphi}_{H^1(\Omega_\e)}^2
			&=\sum_{i=1}^m \lambda_{n+i-1}^\e(a_i^\e-b_i^\e)^2\leq
			\lambda_{n+m-1}^\e \norm{\Pi_\e^n P_\e(\varphi)-
				\Pi_\e^n\varphi}_{L^2(\Omega_\e)}^2 \\
			&\leq \lambda_{n+m-1}^\e\norm{U_\e^\varphi}_{L^2(\Omega_\e)}^2,
		\end{align*}
		where we also used the fact that $\Pi_\e^n$ is an orthogonal
		projection and the definition of $P_\e$. Hence, from
		\Cref{lemma:L2_decay} it follows that
		\begin{equation*}
			\sup_{\substack{\varphi\in
					E(\lambda_n)\\\|\varphi\|_{L^2(\Omega)}=1}}
			\norm{\Pi_\e^n P_\e(\varphi)-\Pi_\e^n\varphi}_{H^1(\Omega_\e)}^2
			=o(\e^N)\quad\text{as }\e\to 0,
		\end{equation*}
		which, combined with \eqref{eq:eigenfunctions_2}, yields
		\eqref{eq:eigenfunctions_3}.
		
		\smallskip    \noindent\textbf{Step 3:} we claim that
		\begin{equation}\label{eq:eigenfunctions_4}
			\sup_{\substack{\varphi\in
					E(\lambda_n)\\\|\varphi\|_{L^2(\Omega)}=1}}
			\norm{P_\e(\varphi)-\frac{\Pi_\e^n\varphi}
				{\norm{\Pi_\e^n\varphi}_{L^2(\Omega_\e)}}}_{H^1(\Omega_\e)}^2
			=o(\e^N)\quad\text{as }\e\to 0.
		\end{equation}
		Arguing as in the previous step, one can derive the estimate
		\begin{equation*}
			\norm{\Pi_\e^n\varphi-\frac{\Pi_\e^n\varphi}
				{\norm{\Pi_\e^n\varphi}_{L^2(\Omega_\e)}}}_{H^1(\Omega_\e)}^2\leq
			\frac{\lambda_{n+m-1}^\e}{\norm{\Pi_\e^n\varphi}_{L^2(\Omega_\e)}^2}
			\abs{\norm{\Pi_\e^n\varphi}_{L^2(\Omega_\e)}-1}^2
		\end{equation*}
		for every $\varphi\in E(\lambda_n)$ with
		$\|\varphi\|_{L^2(\Omega)}=1$. On the other hand, 
		\eqref{eq:eigenfunctions_3} and \Cref{lemma:L2_decay} yield
		\begin{align*}
			\abs{\norm{\Pi_\e^n\varphi}_{L^2(\Omega_\e)}-1}&\leq
			\abs{\norm{\Pi_\e^n\varphi}^2_{L^2(\Omega_\e)}-1}=
			\left|\int_{\Omega_\e}(\Pi_\e^n\varphi-\varphi) (\Pi_\e^n\varphi+\varphi)\dx+ \int_{\Omega_\e}\varphi^2\dx-1\right|
			\\
			&\leq
			\norm{\varphi-\Pi_\e^n\varphi}_{L^2(\Omega_\e)}\norm{\varphi+\Pi_\e^n\varphi}_{L^2(\Omega_\e)}+\int_{\Sigma_\e}\varphi^2\dx\\
			&\leq 2
			\Big(\norm{P_\e(\varphi)-\Pi_\e^n\varphi}_{L^2(\Omega_\e)}
			+\norm{U_\e^\varphi}_{L^2(\Omega_\e)}\Big)
			+\int_{\Sigma_\e}\varphi^2\dx=o(\e^{N/2})
		\end{align*}
		as $\e\to 0$, uniformly with respect to
		$\varphi\in E(\lambda_n)$ with
		$\|\varphi\|_{L^2(\Omega)=1}$. Combining the above two
		estimates with \eqref{eq:eigenfunctions_3}, we finally obtain
		\eqref{eq:eigenfunctions_4}.
		
		\smallskip    \noindent\textbf{Step 4:} we claim that
		\begin{equation}\label{eq:eigenfunctions_fin}
			\sup_{\substack{\varphi\in
					E(\lambda_n)\\\|\varphi\|_{L^2(\Omega)}=1}}
			\left|
			\norm{\varphi-\frac{\Pi_\e^n\varphi}
				{\norm{\Pi_\e^n\varphi}_{L^2(\Omega_\e)}}}_{H^1(\Omega_\e)}^2
			-\norm{U_\e^\varphi}_{H^1(\Omega_\e)}^2\right|=o(\e^N)
			\quad\text{as }\e\to 0.
		\end{equation}
		Indeed,
		\begin{multline*}
			\norm{\varphi-\frac{\Pi_\e^n\varphi}
				{\norm{\Pi_\e^n\varphi}_{L^2(\Omega_\e)}}}_{H^1(\Omega_\e)}^2
			-\norm{U_\e^\varphi}_{H^1(\Omega_\e)}^2\\=\left(P_\e(\varphi)-
			\frac{\Pi_\e^n\varphi}
			{\norm{\Pi_\e^n\varphi}_{L^2(\Omega_\e)}}, P_\e(\varphi)-
			\frac{\Pi_\e^n\varphi}
			{\norm{\Pi_\e^n\varphi}_{L^2(\Omega_\e)}}
			+2U^\varphi_\e\right)_{H^1(\Omega_\e)},
		\end{multline*}
		so that \eqref{eq:eigenfunctions_fin} follows from the
		Cauchy-Schwarz inequality, taking into account
		\eqref{eq:eigenfunctions_4} and \Cref{lemma:H1_decay}.
		
		\smallskip \noindent The thesis of Theorem
		\ref{thm:eigenfunctions} is finally a consequence of
		\eqref{eq:eigenfunctions_fin} and \Cref{lemma:blowup}.
	\end{proof}

	\section{Splitting of double eigenvalues}\label{sec:splitting}
	
	\noindent
	The present section is devoted to the proof of \Cref{thm:3}. Hence, we
	assume $\Sigma=B_1$.
	
	\begin{proof}[Proof of \Cref{thm:3}]
		From Corollary \ref{cor:spherical}, for every $x\in \Omega$ and
		$i,j=1,\dots,m$ we have
		\begin{equation}\label{eq:asy-diff}
			\e^{-N}\big(
			\lambda_{n+i-1}(\Omega_\e^{x,B_1})-
			\lambda_{n+j-1}(\Omega_\e^{x,B_1})\big)
			=\gamma_j^{x}-\gamma_{i}^{x}+o(1)\quad\text{as }\e\to0,
		\end{equation}
		where $\{\gamma_i^{x}:i=1,2,\dots,m\}$ are the eigenvalues of
		the bilinear form $\mathcal B_{x}$ on $E(\lambda_n)$ defined in
		\eqref{eq:defB}; equivalently, once the $L^2(\Omega)$-orthonormal
		basis $\{\varphi_{n+i-1}:i=1,2,\dots,m\}$ of $E(\lambda_n)$ is fixed,
		$\{\gamma_i^{x}:i=1,2,\dots,m\}$ are the eigenvalues (in
		descending order) of the $m\times m$ real symmetric matrix
		\begin{equation*}
			B_x:=\Big(
			\mathcal B_{x}(\varphi_{n+i-1}, \varphi_{n+j-1})\Big)_{1\leq
				i,j\leq m}.
		\end{equation*}
		The conclusion will follow directly from \eqref{eq:asy-diff} once we
		have proved that
		\begin{equation}\label{eq:claim-thm3}
			\text{$\Gamma=\{x\in\Omega: \gamma_1^{x}=\gamma_2^{x}=\cdots=\gamma_m^{x}\}$
				is relatively closed in $\Omega$ and  $\dim_{\mathcal{H}}\Gamma\leq N-1$}.
		\end{equation}
		To prove claim \eqref{eq:claim-thm3}, we first observe that the
		eigenvalues of a real symmetric square matrix are all equal to each
		other if and only if the matrix is a multiple of the identity matrix.
		Therefore, $x\in\Gamma$ if and only if
		\begin{equation*}
			\mathcal B_{x}(\varphi_{n}, \varphi_{n}) =
			\mathcal B_{x}(\varphi_{n+1}, \varphi_{n+1})=\dots= \mathcal
			B_{x}(\varphi_{n+m-1}, \varphi_{n+m-1}).
		\end{equation*}
		Hence
		\begin{equation*}
			\Gamma=\{x\in\Omega: g_1(x)=g_2(x)=\cdots=g_m(x)\}
		\end{equation*}
		where
		\begin{equation*}
			g_i(x):=\frac{N}{N-1}|\nabla\varphi_{n+i-1}(x)|^2-
			(\lambda_n-1)\varphi_{n+i-1}^2(x) \quad i=1,2,\dots,m.
		\end{equation*}
		By classical regularity results, see e.g. \cite{morrey1958}, the
		eigenfunctions $\varphi_n, \varphi_{n+1},\dots, \varphi_{n+m-1}$ are
		analytic in $\Omega$. Hence the functions
		\begin{equation*}
			x\mapsto g_i(x), \quad i=1,2,\dots,m
		\end{equation*}
		are analytic in $\Omega$.  This implies, first of all, that
		$\Gamma$ is relatively closed in $\Omega$. Moreover, it is
		well-known that the zero set of any real-analytic function on a
		connected open subset of $\R^N$, not identically zero, has
		Hausdorff dimension smaller than or equal to $N-1$: this is known
		in the literature as \emph{Lojasiewicz's stratification theorem},
		see e.g. \cite[Theorem 6.3.3]{primer_analytic} and
		\cite[Proposition 3]{M20}.  Thus, if $g_1-g_2\not\equiv0$ in
		$\Omega$, then
		\begin{equation*}
			\dim_{\mathcal{H}}(\Gamma)\leq\dim_{\mathcal{H}}(\{x\in
			\Omega:g_1(x)-g_2(x)=0\})\leq N-1,
		\end{equation*}
		thus completing the proof of claim \eqref{eq:claim-thm3}. To
		conclude, it is therefore sufficient to prove that
		\begin{equation}
			\label{eq:claim2}
			g_1(x)-g_2(x)\neq 0\quad\text{for some }x\in\Omega.
		\end{equation}
		We prove \eqref{eq:claim2} arguing by contradiction. To this aim
		we assume that
		\begin{equation}\label{eq:thm2_1}
			\frac{N}{N-1}|\nabla\varphi_n(x)|^2-(\lambda_n-1)\varphi_n
			^2(x)
			=\frac{N}{N-1}|\nabla\varphi_{n+1}(x)|^2-(\lambda_n-1)
			\varphi_{n+1}^2 (x)
		\end{equation}
		for every $x\in \Omega$. The rest of the
		proof is devoted to show that this leads to a contradiction.
		
		Let
		\begin{equation*}
			u:=\varphi_n+\varphi_{n+1}\quad\text{and}\quad v:=\varphi_n-\varphi_{n+1},
		\end{equation*}
		so that \eqref{eq:thm2_1} reads as
		\begin{equation}\label{eq:thm2_2}
			\nabla u(x)\cdot\nabla v(x)=(\lambda_n-1)\,\frac{N-1}{N}\,
			u(x)v(x)\quad\text{for all }x\in\Omega.
		\end{equation}
		Moreover $u,v\in E(\lambda_n)$ are orthogonal in $L^2(\Omega)$ and
		belong to $H^1(\Omega)\cap L^\infty(\Omega)$, see
		\cite{Winkert2010}. Hence $uv\in H^1(\Omega)$.
		Moreover, for
		every $\varphi\in C^\infty(\overline{\Omega})$,
		$u\varphi\in H^1(\Omega)$ and $v\varphi\in H^1(\Omega)$; hence we
		may test the equation satisfied by $u$ with $v\varphi$ and the
		equation satisfied by $v$ with $u\varphi$, thus obtaining
		\begin{align*}
			&\int_\Omega (\nabla u\cdot\nabla \varphi)v\dx+\int_\Omega
			(\nabla u\cdot\nabla v) \varphi\dx=(\lambda_n-1)
			\int_\Omega uv \varphi\dx,\\
			&\int_\Omega (\nabla v\cdot\nabla \varphi)u\dx+\int_\Omega
			(\nabla u\cdot\nabla v) \varphi\dx=(\lambda_n-1)
			\int_\Omega uv \varphi\dx;
		\end{align*}
		summing up and using \eqref{eq:thm2_2} we obtain
		\begin{equation}\label{eq:eq-uv}
			\int_\Omega \Big(\nabla (uv)\cdot\nabla
			\varphi+uv\varphi\Big)\dx=
			\Big(\frac{2(\lambda_n-1)}{N}+1\Big) \int_\Omega uv \varphi\dx\quad
			\text{for all }\varphi\in C^\infty(\overline{\Omega}).
		\end{equation}
		By density of $C^\infty(\overline{\Omega})$ in $H^1(\Omega)$ we
		conclude that the above identity is satisfied for all
		$\varphi\in H^1(\Omega)$, i.e. $uv$ satisfies
		\begin{equation*}
			\begin{bvp}
				-\Delta(uv)+uv&=\Lambda_n uv, &&\text{in }\Omega, \\
				\partial_{\nnu}(uv)&=0, &&\text{on }\partial\Omega,
			\end{bvp}
		\end{equation*}
		in a weak sense, where
		\begin{equation*}
			\Lambda_n:=\frac{2}{N}(\lambda_n-1)+1.
		\end{equation*}
		Let us fix any nodal domain $\omega$ of the function $u$, i.e. $\omega$ is a
		connected component of $\Omega\setminus \mathcal N_u$, where
		$\mathcal N_u=\{x\in\Omega\colon u(x)=0\}$ is the nodal set of
		$u$. Without loss of generality, we assume that $u>0$ on $\omega$.
		Let us consider the space
		\begin{equation*}
			\mathcal H_\omega=\{w\in H^1(\omega):
			\mathsf{P}_{0}w\in H^1(\Omega)\}, 
		\end{equation*}
		where
		\begin{equation*}
			\mathsf{P}_{0}w=
			\begin{cases}
				w,&\text{in }\omega,\\
				0,&\text{in }\Omega\setminus\omega,
			\end{cases}
		\end{equation*}
		is the trivial extension of $w$ in $\Omega$. We observe that $\mathcal
		H_\omega$ is a closed subspace of $H^1(\omega)$, hence it is a Hilbert
		space with the same scalar product as $H^1(\Omega)$. Furthermore, we
		claim that
		\begin{equation}\label{eq:claimHomega}
			\text{if $w\in H^1(\Omega)\cap C^0(\Omega)$ and $w(x)=0$ for every
				$x\in \mathcal N_u$, then $w\restr{\omega}\in
				\mathcal H_\omega$}.
		\end{equation}
		In order to prove the claim, we argue as in \cite[Theorem
		9.17]{brezis} and consider a function $G\in~\!\!C^1(\R)$ such that
		\begin{equation*}
			|G(t)|\leq |t|\quad\text{for all }t\in\R\quad\text{and}\quad  
			G(t)=\begin{cases}
				0,&\text{if }|t|\leq 1, \\
				t,&\text{if }|t|\geq 2.
			\end{cases}
		\end{equation*}
		For some  $w\in H^1(\Omega)\cap C^0(\Omega)$ such that $w=0$ on
		$\mathcal N_u$, we introduce the sequence of functions
		\begin{equation*}
			w_k:\omega\to\R,\quad   w_k(x):=\frac{1}{k}G(k w(x)).
		\end{equation*}
		One can easily verify that $w_k\in H^1(\omega)$ and $w_k=0$ in
		some open neighborhood of $\mathcal N_u$. Hence
		$\mathsf{P}_{0}w_k\in H^1(\Omega)$, i.e.
		$w_k\in \mathcal H_\omega$. In addition, reasoning as in
		\cite[Theorem 9.17]{brezis}, one can see that
		$w_k\to w\restr{\omega}$ in $H^1(\omega)$ as
		$k\to\infty$. Therefore, $w\restr{\omega}\in\mathcal H_\omega$,
		thus proving claim \eqref{eq:claimHomega}.
		In particular, \eqref{eq:claimHomega} ensures that
		\begin{equation*}
			u\restr{\omega}\in\mathcal H_\omega\quad\text{and}\quad 
			uv\restr{\omega}\in\mathcal H_\omega.
		\end{equation*}
		Let 
		\begin{equation*}
			\Lambda_1(\omega):=
			\min\left\{ \mathcal R_\omega(w)\colon w\in \mathcal H_\omega
			\setminus\{0\}\right\},
		\end{equation*}
		where
		\begin{equation*}
			\mathcal R_\omega(w):=\frac{\displaystyle\int_\omega(|\nabla w|^2+w^2)\dx }
			{\displaystyle\int_\omega w^2\dx}.
		\end{equation*}
		By standard minimization arguments, $\Lambda_1(\omega)$ is attained by
		some function $w_0\in\mathcal H_\omega\setminus\{0\}$ such that $w_0\geq0$ in
		$\omega$; moreover, such a minimizer $w_0$ necessarily satisfies
		\begin{equation}\label{eq:eqminim}
			\int_\omega (\nabla w_0\cdot \nabla
			\varphi+w_0\varphi)\dx=\Lambda_1(\omega)
			\int_\omega w_0\varphi\dx\quad\text{for every }\varphi\in\mathcal H_\omega,
		\end{equation}
		and, by the Strong Maximum Principle, $w_0>0$ in $\omega$. It is also
		standard to prove that any function
		$w\in \mathcal H_\omega\setminus\{0\}$ such that
		$\mathcal R_\omega(w)=\Lambda_1(\omega)$ solves the variational
		equation \eqref{eq:eqminim},  being either $w>0$ or $w<0$ in $\Omega$;
		futhermore, if $w_1,w_2\in \mathcal H_\omega\setminus\{0\}$ are such
		that $\mathcal R_\omega(w_1)=\mathcal
		R_\omega(w_2)=\Lambda_1(\omega)$, then $w_2=cw_1$ for some $c\in\R\setminus\{0\}$.

		Choosing $\varphi=u\restr{\omega}$ in \eqref{eq:eqminim}, testing
		the equation satisfied by $u$ with $\mathsf{P}_{0}w_0$, and taking the
		difference we obtain that
		\begin{equation*}
			(\Lambda_1(\omega)-\lambda_n)  \int_\omega w_0 u\dx=0,
		\end{equation*}
		and hence
		\begin{equation*}
			\Lambda_1(\omega)=\lambda_n.
		\end{equation*}
		Moreover, the equation satisfied by $u$ tested with
		$\mathsf{P}_{0}u\restr{\omega}$
		implies that  $\mathcal R_\omega\big( u\restr{\omega}\big)=\lambda_n$.
		
		Since $uv\restr{\omega}\in \mathcal H_\omega$, we can test
		\eqref{eq:eq-uv} with $\mathsf{P}_{0}\big(uv\restr{\omega}\big)\in
		H^1(\Omega)$, thus obtaining $\Lambda_n=\mathcal R_\omega
		\big(uv\restr{\omega}\big)$. It follows that
		\begin{equation}\label{eq:compar-lambdas}
			\lambda_n=\Lambda_1(\omega)\leq \mathcal R_\omega
			\big(uv\restr{\omega}\big)=\Lambda_n=\frac{2}{N}(\lambda_n-1)+1
		\end{equation}
		Since $\lambda_n>1$ by \eqref{eq:lambda_n>1}, the above
		inequality directly gives rise to a contradiction if $N\geq3$.  If
		$N=2$, we have $\Lambda_n=\lambda_n$, hence \eqref{eq:compar-lambdas}
		implies that
		$\mathcal R_\omega \big(uv\restr{\omega}\big)=\Lambda_1(\omega)=
		\lambda_n= \mathcal R_\omega \big(u\restr{\omega}\big)$; this implies
		that $uv\restr{\omega}=cu\restr{\omega}$ for some constant
		$c\in\R\setminus\{0\}$, hence $v\equiv c$ in $\omega$, contradicting
		the facts that $-\Delta v+v=\lambda_n v$ in $\omega$ and
		$\lambda_n\neq1$.
	\end{proof}

	\section{Numerical exploration of the general case}\label{sec:numerics}
	
	\noindent
	We conclude this work with a brief numerical exploration that
	complements the theoretical results obtained in the Neumann case. In
	particular, we examine the behavior of multiple eigenvalues when the
	hole is not spherical. This section is intended to offer a first
	indication of what may occur in situations beyond the scope of our
	analysis, and to suggest possible directions for future work.
	
	In the following experiments, we compute the eigenvalues using the PDE
	Toolbox of the commercial software MATLAB. The underlying numerical
	scheme is the Finite Element Method (FEM), based on triangular mesh
	elements and linear basis functions. The exact numerical parameters
	used in each experiment are specified in the respective subsections.
	
	\subsection{Validation Experiment}
	We begin by validating the numerical scheme. To this end, we consider
	the setting of Corollary \ref{cor:thm3}. In particular, we consider
	\begin{equation}
		\label{eq:rectangle}
		\Omega = (-4,4)\times(-2,2) \subset \R^2
	\end{equation}
	and 
	\begin{equation}
		\label{eq:center}
		\Sigma_\e^{x_0} = x_0 + \e B_1, \quad \text{where}\quad x_0 = (2.1,0.1).
	\end{equation}
	In this setting, the eigenvalues of the unperturbed problem can be
	computed explicitly. We have
	\begin{equation}
		\label{eq:2d_computed_eig}
		\lambda_{m,n}(\Omega) = 1 + \frac{\pi^2}{16}\left(\frac{m^2}{4} +
		n^2\right),
		\quad m,n\in\N.
	\end{equation}
	We focus on the first 10 eigenvalues, counted with multiplicity. In
	particular, we observe that
	\[\lambda_{0,1} = \lambda_{2,0} \sim 1.616\quad \text{and} \quad
	\lambda_{0,2} = \lambda_{4,0} \sim 3.467\]
	have multiplicity equal to 2. 
	
	At this point, taking 
	\[
	\e = 2^{-i}, \quad  i = 0,\ldots,7,
	\]
	and setting the maximal size of the finite elements to be equal to
	$h=0.05$, we numerically solve the eigenvalue problem for both the
	unperturbed and perturbed PDE. The results are summarized in Figure
	\ref{fig:validation}.
	\begin{figure}[h!]
		\centering
		\includegraphics[width=0.75\linewidth]{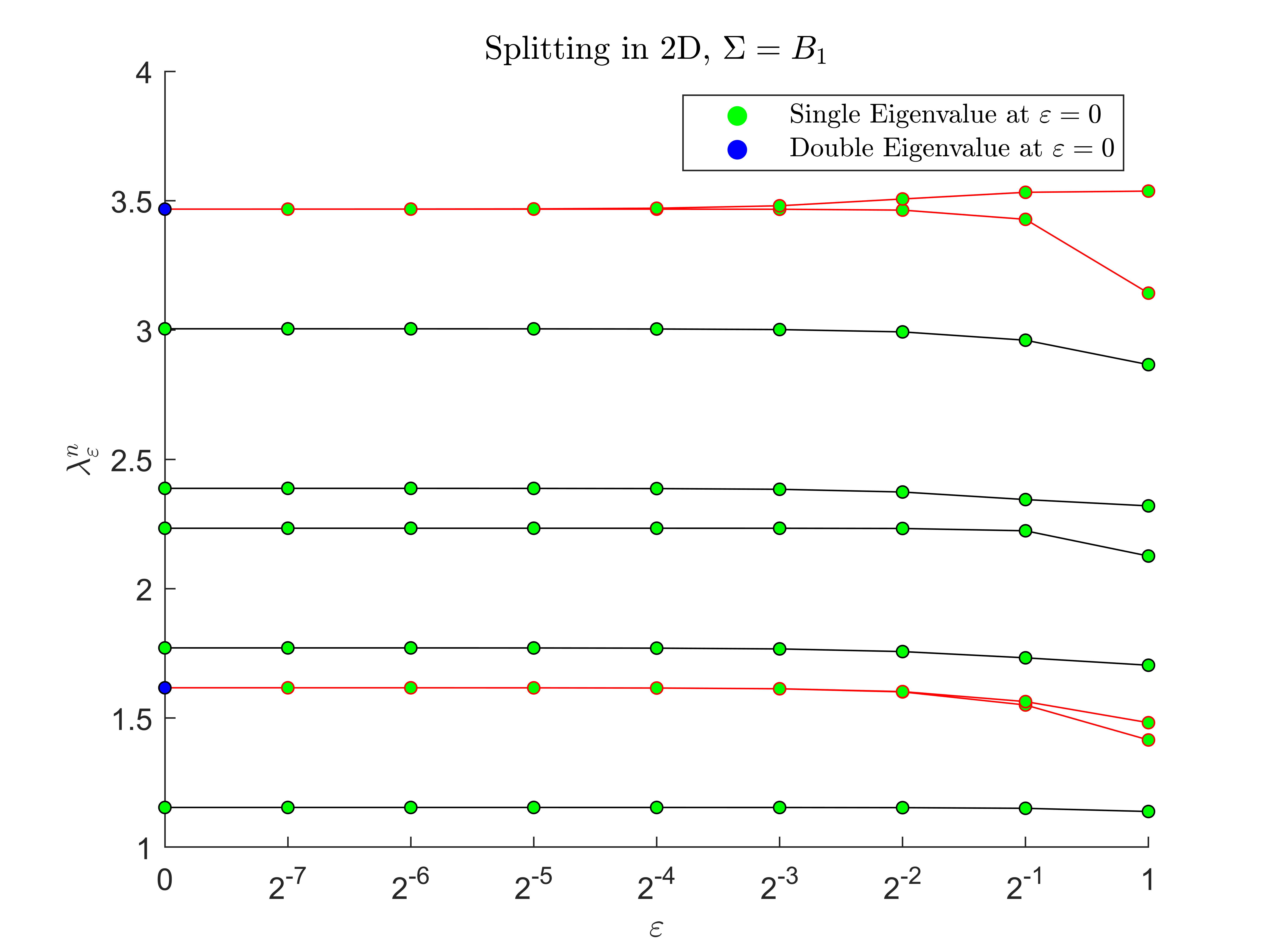}
		\caption{\footnotesize{The computed eigenvalues for different
				values of $\e$ in the case $\Sigma=B_1$. For $\e=0$ we have
				the eigenvalues of the unperturbed problem. The multiple
				eigenvalues are colored in blue.}}
		\label{fig:validation}
	\end{figure}
	
	As we can see, the two eigenvalues with multiplicity 2 at $\e = 0$
	(colored in blue) split into different branches as soon as $\e>0$, as
	predicted by Corollary \ref{cor:thm3}. Note that, even though we could
	have hard-coded the value of the unperturbed eigenvalues, we have
	preferred to also compute them numerically, to further ascertain the
	precision of the solver.
	
	\subsection{The 2D Non-spherical Case}
	
	We now move to an experiment in a setting not covered by our
	theoretical results. The setting of the problem at $\e = 0$ is the
	same outlined in the previous subsection: the domain $\Omega$ is the
	rectangle \eqref{eq:rectangle}, and the eigenvalues of the unperturbed
	problem are described by \eqref{eq:2d_computed_eig}. We also position
	the center $x_0$ of the hole as in \eqref{eq:center}. However, in this
	case $\Sigma$ is a five-pointed star, see Figure
	\ref{fig:mesh}. In particular, the chosen shape is not spherical,
	nor convex.
	
	\begin{figure}[h!]
		\centering
		\includegraphics[width=0.58\linewidth]{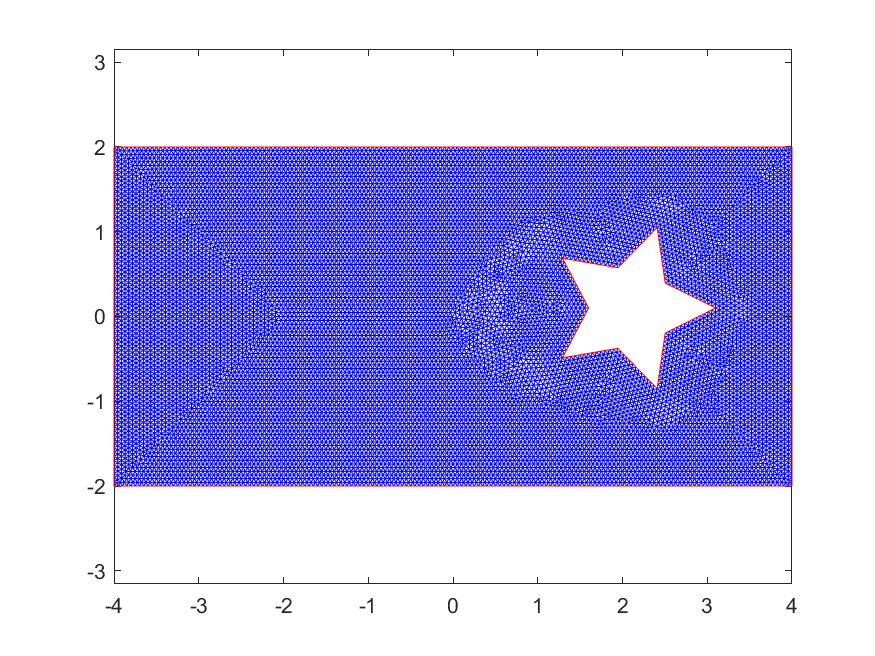}
		\caption{\footnotesize{The setting of the second experiment. In
				this case, the hole is shaped as a five-pointed star. Here, we
				can see the FEM mesh used by the numerical solver in the case
				$\e = 1$.}}
		\label{fig:mesh}
	\end{figure}
	
	As before, we compute the eigenvalues of the perturbed problem for
	increasingly small values of the diameter
	\[
	\e = 2^{-i}, \quad  i = 0,\ldots,7.
	\]
	The results are shown in Figure \ref{fig:exp1}.  The behavior is
	very similar to that observed in the spherical case, suggesting
	that splitting should happen also for non-spherical holes.
	\begin{figure}[h!]
		\centering
		\includegraphics[width=0.75\linewidth]{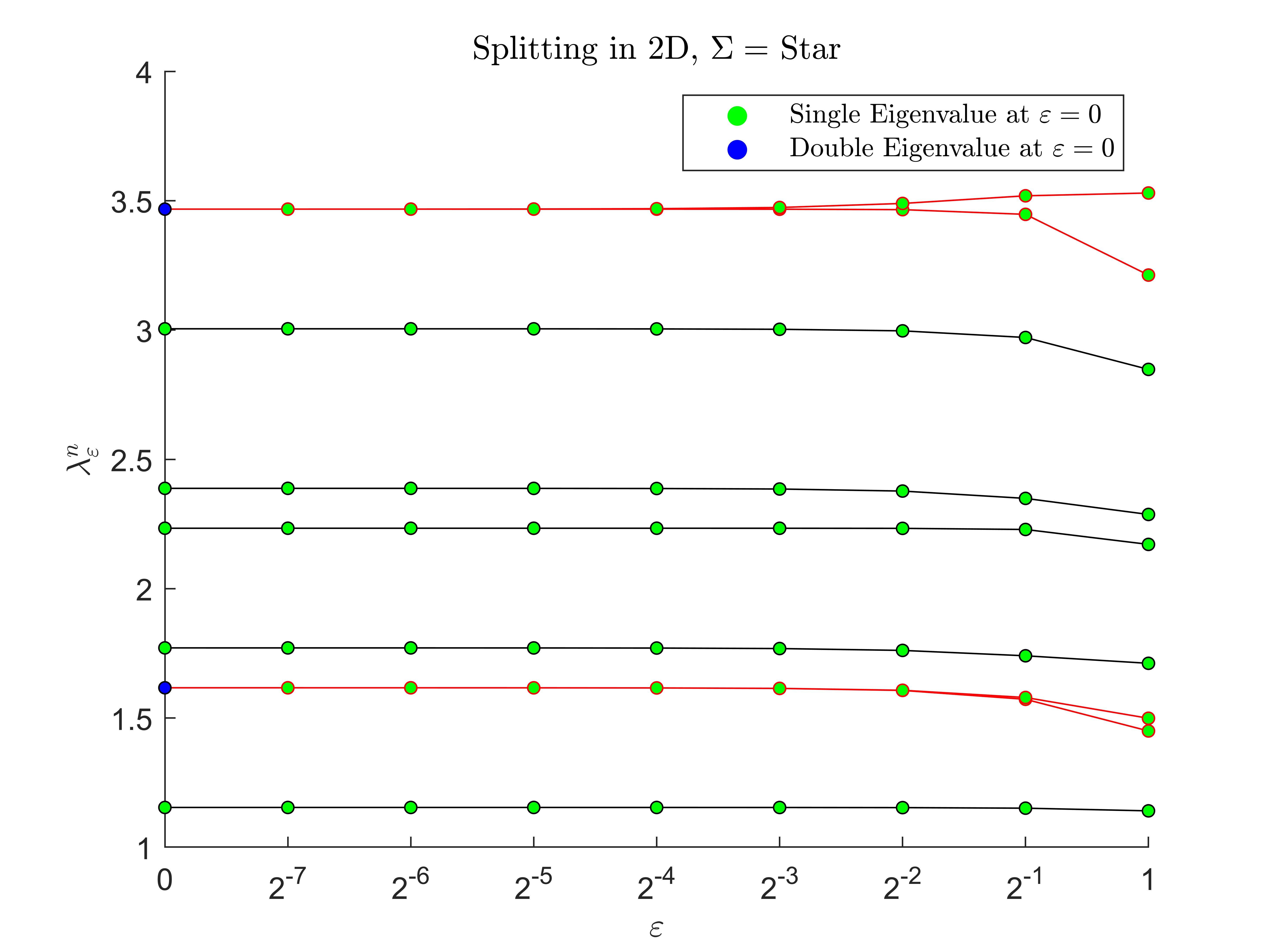}
		\caption{\footnotesize{The computed eigenvalues for different
				values of $\e$ in the case $\Sigma$ is a five-pointed
				star. For $\e=0$ we have the eigenvalues of the unperturbed
				problem. The multiple eigenvalues are colored in blue.}}
		\label{fig:exp1}
	\end{figure}
	
	\subsection{The 3D Non-spherical Case}
	
	To conclude, we consider a case in three dimensions. This time, we
	consider the box-shaped domain
	\[
	\Omega = (-1,1) \times (-2,2) \times (-3,3).
	\]
	Accordingly, the eigenvalues of the unperturbed problem follow the
	formula
	\[
	\lambda_{\ell,m,n}(\Omega) = 1 + \frac{\pi^2}{4}\left(\ell^2 +
	\frac{m^2}{4} + \frac{n^2}{9}\right), \quad \ell, m, n \in \N. 
	\]
	In particular, the spectrum includes eigenvalues with
	multiplicities $2$ and $3$.  In our experiment we will focus on
	\begin{align*}
		\lambda_{1,0,0} = \lambda_{0,2,0} = \lambda_{0,0,3} &= 1+ \frac{\pi^2}{4}\sim 3.467, \\ 
		\lambda_{1,0,1} = \lambda_{0,2,1} &\sim 3.741,\\
		\lambda_{1,1,0} = \lambda_{0,1,3} &\sim 4,084.
	\end{align*}
	The hole shape $\Sigma$ is given, in this case, by a hollow cylinder,
	with outer ratio $R$, inner ratio $r = R/2$ and height $h = 2R$. We
	choose the base value $R=0.5$ and let
	\[
	\e = 0.05 \cdot i, \quad i = 0,\ldots,10.
	\]
	The reason for this linear scaling, rather than the logarithmic one of
	the previous plots, is due to the FEM scheme. Indeed, using a
	logarithmic scaling of $\e$ would make the hole too small for the mesh
	to see at reasonable mesh sizes. An example of the perforated domain
	used in this experiment can be observed in Figure \ref{fig:exp3d}.
	The results are shown in Figure \ref{fig:exp3d_2}. We observe not only
	the splitting of the double eigenvalues, but also that of a triple
	eigenvalue into three distinct branches. Interestingly, the splitting
	of the triple eigenvalue is considerably more pronounced than that of
	the double ones.
	
	\begin{figure}[h!]
		\centering
		\includegraphics[width=0.7\linewidth]{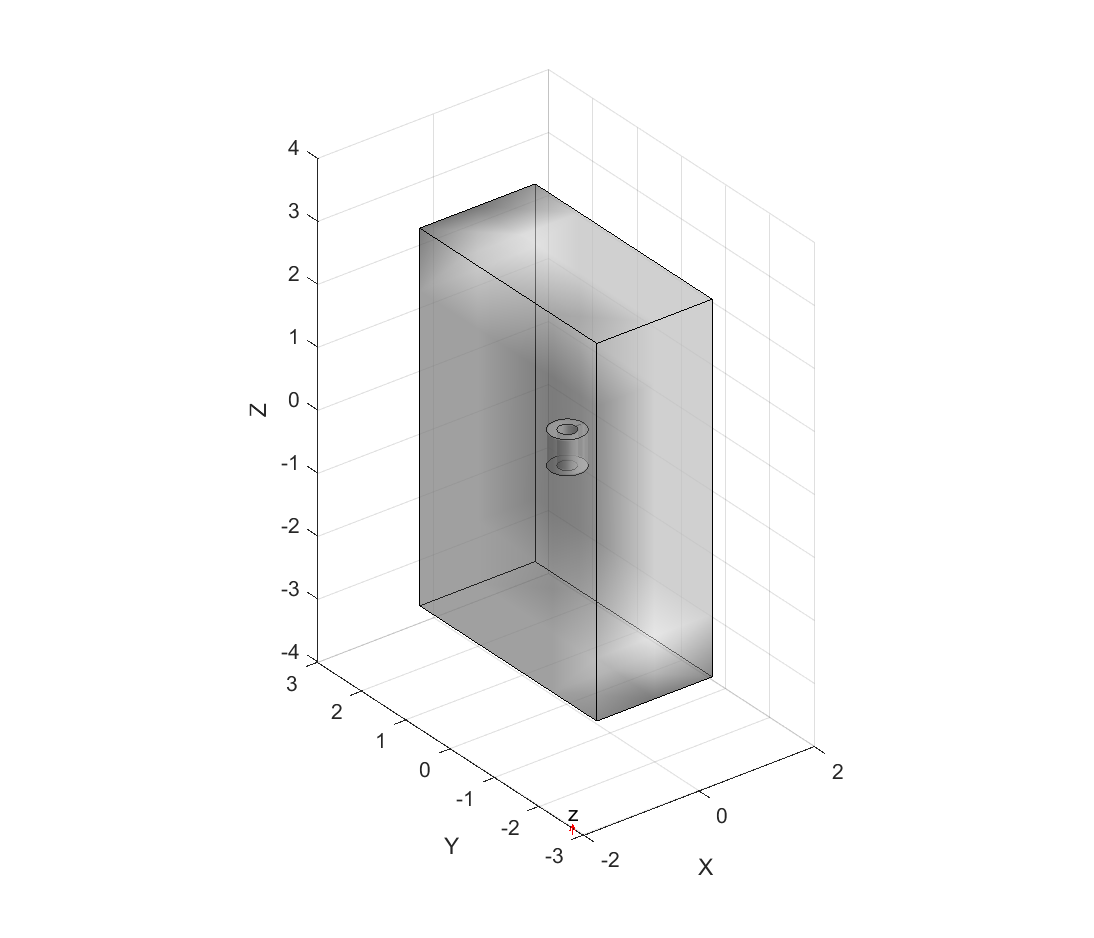}
		\caption{The perforated domain used in this numerical experiment. In this case $\e=0.25$.}
		\label{fig:exp3d}
	\end{figure}

	\begin{figure}[h!]
		\centering
		\includegraphics[width=0.75\linewidth]{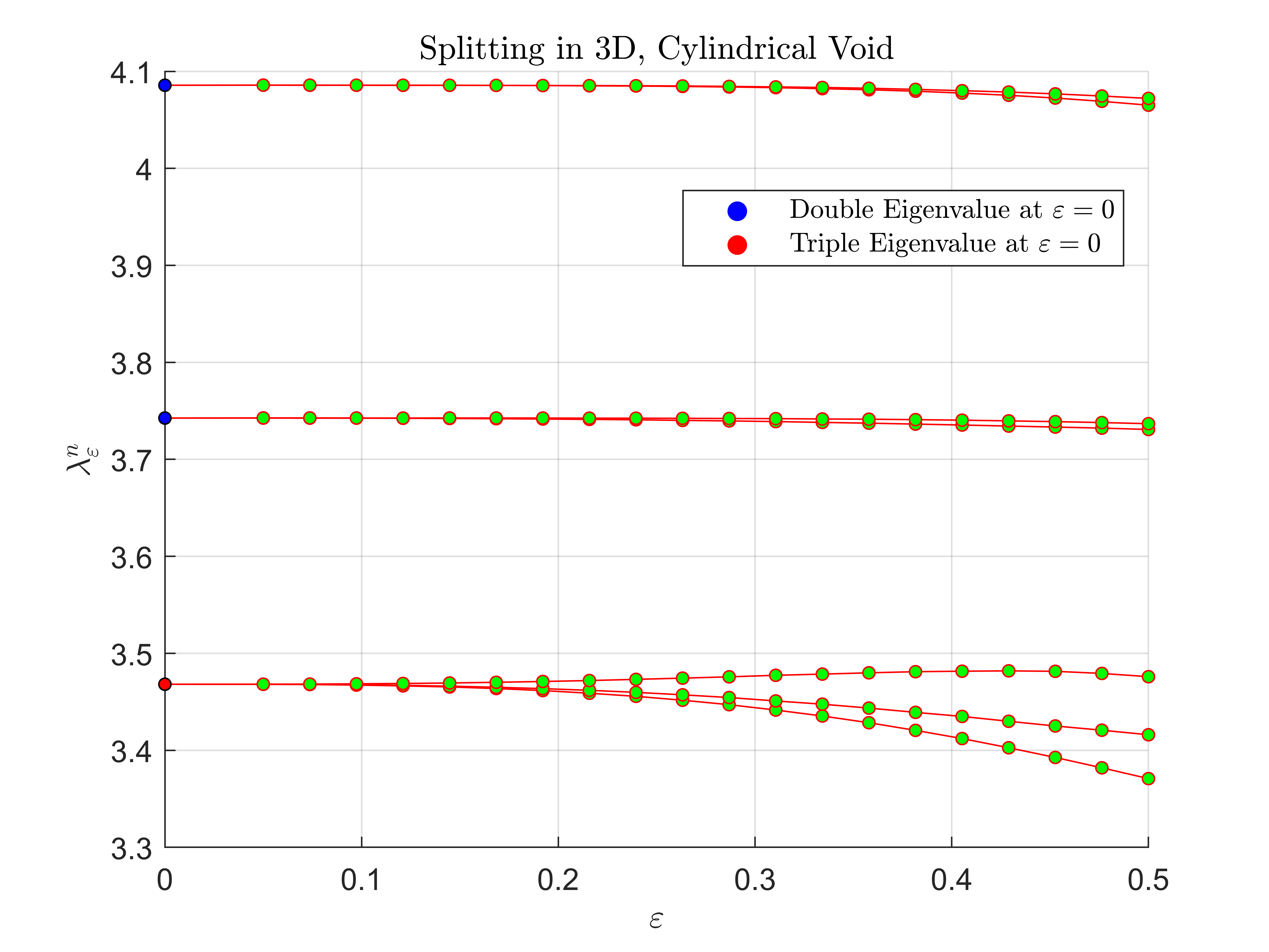}
		\caption{\footnotesize{The computed eigenvalues for different
				values of $\e$ in the case $\Sigma$ is a hollow cilinder. For
				$\e=0$ we have some multiple eigenvalues of the unperturbed problem. The
				double eigenvalues are colored in blue. The triple eigenvalue
				is colored in red.}}
		\label{fig:exp3d_2}
	\end{figure}

	\appendix
	\section{The Dirichlet Setting}\label{appendix}
	
	\noindent
	In this short appendix, we investigate the Dirichlet
	setting. Exploiting results from references
	\cite{AFHL,ALM24,FNO1}, we recover known splitting properties for
	the Dirichlet Laplacian, see for instance
	\cite{flucher,dabrwoski1}. We consider the eigenvalue problem for the
	Laplacian in a bounded connected open set $\Omega\sub\R^N$, namely
	\begin{equation}\label{eq:dirichlet}\tag{$D$}
		\begin{bvp}
			-\Delta u&=\lambda u, &&\text{in }\Omega, \\
			u&=0, &&\text{on }\partial\Omega.
		\end{bvp}
	\end{equation}
	Problem \eqref{eq:dirichlet} is known to have a sequence of
	eigenvalues $\{\lambda_k^{\textup{D}}(\Omega)\}_{k\in\N}$. We
	denote by $E(\lambda_k^{\textup{D}}(\Omega))$ the eigenspace
	corresponding to the eigenvalue $\lambda_k^{\textup{D}}(\Omega)$.
	
	For some fixed $n\in\N$, let $m\in\N\setminus\{0\}$ be the
	multiplicity of $\lambda_n^{\textup{D}}(\Omega)$, so that
	\begin{equation*}
		\lambda_{n-1}^{\textup{D}}(\Omega)<\lambda_n^{\textup{D}}(\Omega)=
		\cdots=\lambda_{n+m-1}^{\textup{D}}(\Omega)
		<\lambda_{n+m}^{\textup{D}}(\Omega).
	\end{equation*}
	Let $\{u_{n+i-1}\}_{i=1,\dots,m}$ be a corresponding eigenbasis of
	$E(\lambda_n^{\textup{D}} (\Omega))$, assumed to be
	$L^2(\Omega)$-orthonormal.
	Let
	\begin{equation*}
		\mathcal N(\lambda_n^{\textup{D}} (\Omega))=
		\big\{x\in\Omega:u_{n+i-1}(x)=0\text{ for all }i=1,\dots,m\big\}
	\end{equation*}
	be the intersection of the nodal sets of all the eigenfunctions
	$\{u_{n+i-1}\}_{i=1,\dots,m}$ of the basis. We observe that
	$\mathcal N(\lambda_n^{\textup{D}} (\Omega))$ does not depend on
	the choice of the basis, has zero Lebesgue measure, and has 
	Hausdorff dimension at most $N-1$.
	
	For any compact set $K\subset \Omega$ and $u\in H^1_0(\Omega)$, we denote by 
	\begin{equation*}
		\mathop{\rm cap}\nolimits_{\Omega}(K,u):=\min\left\{\int_\Omega |\nabla v|^2\dx
		\colon v\in H^1_0(\Omega)\text{ and }v-u\in H^1_0(\Omega\setminus K)\right\}
	\end{equation*}
	the $u$-capacity of $K$ relative to $\Omega$ and by
	$V_{K,u}\in H^1_0(\Omega)$ the corresponding capacitary potential,
	i.e. the unique function attaining the above minimum.  Choosing
	$u=\eta_K$ for some function
	$\eta_K\in C^\infty_{\rm c}(\Omega)$ such that $\eta_K\equiv1$ in a
	neighborhood of $K$, we recover the classical notion of capacity
	$\mathop{\rm cap}\nolimits_{\Omega}(K):= \mathop{\rm cap}\nolimits_{\Omega}(K,\eta_K)$.

	We recall the following result from \cite[Theorem 1.9]{ALM2022}, see
	also \cite{Courtois1995}.
	\begin{theorem}[\cite{ALM2022}]\label{t:ALM22}
		Let $N\geq 2$. Let $K\subset \Omega$ be a compact set such that
		$\mathop{\rm cap}_\Omega(K)=0$ and let $\{K_\e\}_{\e>0}$ be a family of
		compact subsets of $\Omega$ which is concentrating to $K$, that is:
		for any open set $U\sub\R^N$ with $K\subset U$, there exists $\e_U>0$
		such that $K_\e\subset U$ for all $\e<\e_U$. Then, for every $i=1,\dots,m$,
		\begin{equation*}
			\lambda_{n+i-1}^{\textup{D}}(\Omega\setminus K_\e)=
			\lambda_n^{\textup{D}}(\Omega)+\mu_i^\e+o(\chi_\e^2)\quad\text{as }\e\to 0,
		\end{equation*}
		where $\{\mu_i^\e\}_{i=1,\dots,m}$ are the eigenvalues of the bilinear form
		\begin{equation}\label{eq:bfd}
			r_\e(u,v):=\int_{\Omega}\nabla V_{K_\e,u}\cdot
			\nabla V_{K_\e,v}\dx-\lambda_n^{\textup{D}}(\Omega)\int_\Omega V_{K_\e,u}V_{K_\e,v}\dx
		\end{equation}
		defined for $u,v\in E(\lambda_n^{\textup{D}}(\Omega))$, and
		\begin{equation*}
			\chi_\e^2:=\sup\left\{\mathop{\rm cap}\nolimits_\Omega(K_\e,u)\colon u\in
			E(\lambda_n^{\textup{D}}(\Omega)),~\norm{u}_{L^2(\Omega)}=1\right\}.
		\end{equation*}
	\end{theorem}
	Let $\Sigma\sub\R^N$ be any bounded non-empty open set and
	$\e_0>0$ be as in \eqref{eq:choose-eps-zero}. For every
	$x_0\in\Omega$ and $\e\in (0,\e_0)$, let $\Sigma_\e^{x_0}$ be
	defined in \eqref{eq:delSigma-eps}. From \cite[Proposition
	1.5]{AFHL} it follows that, for every
	$u\in E(\lambda_n^{\textup{D}}(\Omega))$ and $x_0\in\Omega$,
	\begin{equation}\label{eq:dir1}
		\mathop{\rm cap}\nolimits_{\Omega}(\overline{\Sigma_\e^{x_0}},u)=u^2(x_0)
		\mathop{\rm cap}\nolimits_{\Omega}(\overline{\Sigma_\e^{x_0}})
		+o(\mathop{\rm cap}\nolimits_{\Omega}(\overline{\Sigma_\e^{x_0}}))\quad\text{as
		}\e\to 0.
	\end{equation}
	Furthermore, by \cite[Corollary A.2]{AFHL} we have  
	\begin{equation}\label{eq:dir2}
		\left\|V_{\overline{\Sigma_\e^{x_0}},u}\right\|^2_{L^2(\Omega)}\dx=o(\mathop{\rm cap}\nolimits_{\Omega}
		(\overline{\Sigma_\e^{x_0}})) \quad\text{as
		}\e\to 0.
	\end{equation}
	If $N=2$, \cite[Proposition 1.6]{AFHL} yields
	\begin{equation}\label{eq:dir3}
		\mathop{\rm
			cap}\nolimits_{\Omega}(\overline{\Sigma_\e^{x_0}})
		=\frac{2\pi}{|\log\e|}+o\left(\frac1
		{|\log\e|}\right)\quad\text{as }\e\to0,
	\end{equation}
	whereas, if $N\geq 3$, we have
	\begin{equation}\label{eq:dir4}
		\mathop{\rm cap}\nolimits_{\Omega}(\overline{\Sigma_\e^{x_0}})=
		\e^{N-2}
		\mathop{\rm cap}\nolimits_{\R^N}(\overline{\Sigma})+o(\e^{N-2})\quad\text{as }\e\to0,
	\end{equation}
	in view of \cite[Theorem 2.14]{FNO1} and \cite[Appendix A]{ALM24}, where
	$\mathop{\rm cap}\nolimits_{\R^N}(\overline{\Sigma})$ is the standard
	Newtonian capacity of $\overline{\Sigma}$, i.e.
	\begin{equation*}
		\mathop{\rm cap}\nolimits_{\R^N}(\overline{\Sigma})=
		\inf\left\{\int_{\R^N}
		|\nabla u|^2\dx\colon u
		\in \mathcal{D}^{1,2}(\R^N),\  u\equiv 1
		\text{ in a neighborhood of }\overline{\Sigma}\right\}.
	\end{equation*}
	Let $r_\e$ be the bilinear form in \eqref{eq:bfd} with
	$K_\e=\overline{\Sigma_\e^{x_0}}$ for some $x_0\in\Omega$.  Combining
	\eqref{eq:dir1}, \eqref{eq:dir2}, \eqref{eq:dir3}, and
	\eqref{eq:dir4}, we obtain 
	\begin{align*}
		r_\e(u,u)&=u^2(x_0)\mathrm{cap}_{\Omega}(\overline{\Sigma_\e^{x_0}})
		+o(\mathrm{cap}_{\Omega}(\overline{\Sigma_\e^{x_0}}))\\
		&=
		\begin{cases}
			\dfrac{2\pi\, u^2(x_0)}{|\log\e|}+o\left(\dfrac1
			{|\log\e|}\right),&\text{if }N=2,\\[10pt]
			\e^{N-2} \mathop{\rm cap}\nolimits_{\R^N}(\overline{\Sigma})\,
			u^2(x_0)+o(\e^{N-2}) ,&\text{if }N\geq 3,
		\end{cases}
	\end{align*}
	as $\e\to0$. Since, by bilinearity,
	$r_\e(u,v)=\frac14\big(r_\e(u+v,u+v)
	-r_\e(u-v,u-v)\big)$, we then deduce that, as $\e\to0$,
	\begin{equation}\label{eq:eqReD}
		r_\e(u,v)=
		\begin{cases}
			|\log\e|^{-1}\mathcal{Q}_{x_0,\Sigma}(u,v)+o(|\log\e|^{-1}),&\text{if }N=2,\\
			\e^{N-2}\mathcal{Q}_{x_0,\Sigma}(u,v)+o(\e^{N-2}),&\text{if }N\geq3,
		\end{cases}
	\end{equation}
	for every $u,v\in E(\lambda_n^{\textup{D}}(\Omega))$, 
	where
	\begin{align}\label{eq:limQFD}
		&\mathcal{Q}_{x_0,\Sigma}: E(\lambda_n^{\textup{D}}(\Omega))\times
		E(\lambda_n^{\textup{D}}(\Omega))\to \R,\\
		&\notag  \mathcal{Q}_{x_0,\Sigma}(u,v)=
		\begin{cases}
			2\pi \,u(x_0)v(x_0),&\text{if }N=2,\\
			\mathop{\rm cap}\nolimits_{\R^N}(\overline{\Sigma}) \,
			u(x_0)v(x_0),&\text{if }N\geq3.
		\end{cases}
	\end{align}
	Theorem \ref{t:ALM22}
	can be then restated in the following more precise version.

	\begin{proposition}\label{prop:dir_expansion}
		For any $x_0\in\Omega$, let $\{\zeta_i^{x_0,\Sigma}\}_{i=1,\dots,m}$
		be the eigenvalues (in ascending order) of the quadratic form
		$\mathcal{Q}_{x_0,\Sigma}$ defined in \eqref{eq:limQFD}. Then, as $\e\to0$,
		\begin{equation*}
			\lambda_{n+i-1}^{\textup{D}}(\Omega_\e^{x_0,\Sigma})=
			\begin{cases}
				\lambda_n^{\textup{D}}(\Omega)+|\log\e|^{-1}\zeta_i^{x_0,\Sigma}
				+o\left(|\log\e|^{-1}\right),&\text{if }N=2,\\[7pt]
				\lambda_n^{\textup{D}}(\Omega)+\e^{N-2}\zeta_i^{x_0,\Sigma}
				+o\left(\e^{N-2}\right),&\text{if }N\geq 3,
			\end{cases}
		\end{equation*}
		for every $i=1,2,\dots,m$.
	\end{proposition}
	\begin{proof}
		The conclusion follows from Theorem \ref{t:ALM22} taking into
		account
		\eqref{eq:eqReD}--\eqref{eq:limQFD}.
	\end{proof}

	The computation of the eigenvalues $\zeta_i^{x_0,\Sigma}$ then
	gives the following result.
	\begin{theorem}\label{t:split-dir-m}
		Let $N\geq2$.
		\begin{enumerate}
			\item[\rm (i)] For every $x_0\in \Omega$ we have, as $\e\to0$,
			\begin{equation*}
				\lambda_{n+m-1}^{\textup{D}}(\Omega_\e^{x_0,\Sigma})=
				\begin{cases}
					\lambda_n^{\textup{D}}(\Omega)+|\log\e|^{-1}
					2\pi\left(\sum_{i=1}^mu_{n+i-1}^2(x_0)\right)
					+o\left(|\log\e|^{-1}\right),&\text{if }N=2,\\[7pt]
					\lambda_n^{\textup{D}}(\Omega)+\e^{N-2}\mathop{\rm cap}
					\nolimits_{\R^N}(\overline{\Sigma})\left(\sum_{i=1}^mu_{n+i-1}^2(x_0)\right)
					+o\left(\e^{N-2}\right),&\text{if }N\geq 3,
				\end{cases}
			\end{equation*}
			and, for every $i=1,2,\ldots,m-1$, 
			\begin{equation*}
				\lambda_{n+i-1}^{\textup{D}}(\Omega_\e^{x_0,\Sigma})=
				\begin{cases}
					\lambda_n^{\textup{D}}(\Omega)+o\left(|\log\e|^{-1}\right),&\text{if }N=2,\\[7pt]
					\lambda_n^{\textup{D}}(\Omega)+o\left(\e^{N-2}\right),&\text{if }N\geq 3.
				\end{cases}
			\end{equation*}
			\item[\rm (ii)] For every
			$x_0\in \Omega\setminus \mathcal N(\lambda_n^{\textup{D}}
			(\Omega))$, there exists $\e_{\tu{s}}=\e_{\tu{s}}(x_0)>0$ such that
			the eigenvalue
			$\lambda_{n+m-1}^{\textup{D}}(\Omega_\e^{x_0,\Sigma})$ is simple and
			the eigenvalues
			$\{\lambda_{n+i-1}^{\tu{D}}(\Omega_\e^{x_0,\Sigma}(x_0))\}_{i=1}^{m-1}$
			have multiplicity strictly lower than $m$ for all
			$0<\e\leq \e_{\tu{s}}$. 
		\end{enumerate}
	\end{theorem}
	\begin{proof}
		The result follows from Proposition \ref{prop:dir_expansion} after
		observing that the $m\times m$ real symmetric matrix
		\begin{equation*}
			\Big(
			u_{n+i-1}(x_0)\cdot u_{n+j-1}(x_0)\Big)_{1\leq
				i,j\leq m}
		\end{equation*}
		has eigenvalues $0$ (with multiplicity $m-1$) and
		$\sum_{i=1}^mu_{n+i-1}^2(x_0)$ (with multiplicity $1$).    
	\end{proof}

	As a direct byproduct of Theorem
	\ref{t:split-dir-m} we have the following genericity result for the splitting
	of double Dirichlet eigenvalues.
	
	\begin{corollary}
		Let $N\geq2$ and $m=2$.  There exists a relatively closed set
		$\mathcal N\sub\Omega$ such that $\dim_{\mathcal{H}}\mathcal N\leq N-1$
		and the following holds: for every
		$x_0\in \Omega\setminus \mathcal N$ there exists
		$\e_{\tu{s}}=\e_{\tu{s}}(x_0)>0$ such that
		$\lambda_n^{\tu{D}}(\Omega_\e^{x_0,B_1}(x_0))
		<\lambda_{n+1}^{\tu{D}}(\Omega_\e^{x_0,B_1}(x_0))$ for all
		$0<\e\leq \e_{\tu{s}}$. In particular, for every
		$x_0\in \Omega\setminus \mathcal N$,
		$\lambda_n^{\tu{D}}(\Omega_\e^{x_0,B_1})$ and
		$\lambda_{n+1}^{\tu{D}}(\Omega_\e^{x_0,B_1})$ are simple, provided
		$\e$ is sufficiently small.
	\end{corollary}
	
	\section*{Acknowledgements}
	\noindent
	V. Felli and R. Ognibene are members of GNAMPA-INdAM.
	R. Ognibene was partially supported by the European Research Council
	(ERC), through the European Union’s Horizon 2020 project ERC VAREG -
	Variational approach to the regularity of the free boundaries (grant
	agreement No. 853404).  V. Felli is supported by the MUR-PRIN
	project no. 20227HX33Z ``Pattern formation in nonlinear phenomena''
	granted by the European Union - Next Generation EU.  The
	authors would like to thank prof. B. Velichkov for helpful discussions.

	\bibliographystyle{aomalpha}
	
	\bibliography{biblio}
	
\end{document}